\documentclass[11pt,a4paper]{amsart}
\usepackage{amssymb,amsmath,epsfig,graphics,mathrsfs}

\usepackage{graphicx}

\usepackage{fancyhdr}
\usepackage{hyperref}
\hypersetup{
 colorlinks   = true,
 urlcolor     = blue,
 linkcolor    = blue,
 citecolor   = red ,
 bookmarksopen=true
}

\usepackage{amsmath}
\usepackage{amsfonts}
\usepackage{amssymb}
\usepackage{amsthm}
\usepackage{epsfig,graphics,mathrsfs}
\usepackage{graphicx}

\usepackage{latexsym} 
\usepackage{amssymb,amsmath}
\usepackage{amsthm}
\usepackage{longtable,booktabs,setspace} 
\usepackage{url}

\usepackage[usenames, dvipsnames]{color} 

\usepackage{hyperref}

\newcommand*\MSC[1][1991]{\par\leavevmode\hbox{%
\textit{\,\,\,\,\, #1 Mathematical subject classification:\ }}}
\newcommand\blfootnote[1]{%
  \begingroup
  \renewcommand\thefootnote{}\footnote{#1}%
  \addtocounter{footnote}{-1}%
  \endgroup
}

\def \phi {\varphi}

\def \R {\mathbb{R}}

\def \G{\Gamma}

\newcommand{\Rn}{\mathbb R^n}

\newcommand{\Hn}{\mathbb H^n}

\newcommand{\p}{\partial}

\newcommand{\bG}{\mathbb {G}}
\newcommand{\bg}{\mathfrak g}

\newcommand{\la}{\lambda}

\numberwithin{equation}{section}

\newcommand{\beq}{\begin{equation}}
\newcommand{\bea}[1]{\begin{array}{#1} }
\newcommand{\eeq}{ \end{equation}}
\newcommand{\ea}{ \end{array}}

\newcommand{\ve}{\varepsilon}

\newcommand{\Lp}{L^p}

\newtheorem{theorem}{Theorem}[section]
\newtheorem{lemma}[theorem]{Lemma}
\newtheorem{proposition}[theorem]{Proposition}
\newtheorem{corollary}[theorem]{Corollary}
\newtheorem{remark}[theorem]{Remark}
\newtheorem{definition}[theorem]{Definition}

\DeclareMathOperator{\sech}{sech}

\numberwithin{equation}{section}

\begin{document}

\title[Feeling the heat, etc.]{Feeling the heat in a group of Heisenberg type}

{\blfootnote{\MSC[2020]{35K08, 35R11, 53C18}}}
\keywords{Extension problems. Conformal invariance. Fundamental solutions}

\date{}

\begin{abstract}
In this paper we use the heat equation in a group of Heisenberg type $\bG$ to provide a unified treatment of the two very different extension problems for the time independent pseudo-differential operators $\mathscr L^s$ and $\mathscr L_s$, $0< s\le 1$. Here, $\mathscr L^s$ is the fractional power of the horizontal Laplacian, and $\mathscr L_s$ is the conformal fractional power of the horizontal Laplacian on $\bG$. One of our main objective is compute explicitly the fundamental solutions of these nonlocal operators by a new approach exclusively based on partial differential equations and semigroup methods. When $s=1$ our results recapture the famous fundamental solution found by Folland and generalised by Kaplan.  
\end{abstract}

\dedicatory{``I have been told that when the work on the first successful atomic pile was being done at the University of Chicago, a copy of Watson's book was chained to a table and always open." \\\emph{From R. Askey's review for the AMS of the second edition of G. N. Watson's classical treatise on Bessel functions}}

\author{Nicola Garofalo}

\address{Dipartimento d'Ingegneria Civile e Ambientale (DICEA)\\ Universit\`a di Padova\\ Via Marzolo, 9 - 35131 Padova,  Italy}
\vskip 0.2in
\email{nicola.garofalo@unipd.it}

\thanks{The first author was supported in part by a Progetto SID (Investimento Strategico di Dipartimento) ``Non-local operators in geometry and in free boundary problems, and their connection with the applied sciences", University of Padova, 2017. Both authors are supported in part by a Progetto SID: ``Non-local Sobolev and isoperimetric inequalities", University of Padova, 2019.}

\author{Giulio Tralli}
\address{Dipartimento d'Ingegneria Civile e Ambientale (DICEA)\\ Universit\`a di Padova\\ Via Marzolo, 9 - 35131 Padova,  Italy}
\vskip 0.2in
\email{giulio.tralli@unipd.it}

\maketitle

\tableofcontents

\section{Introduction}\label{S:intro}

In Theorem 2 of his 1973 note \cite{Fo} Folland proved the following remarkable result:

\noindent \textbf{Theorem.}\  
\emph{The fundamental solution with pole at the group identity of the horizontal Laplacian $-\mathscr L$ in the Heisenberg group $\Hn$ is given by
\begin{equation}\label{T:folland}
\mathscr E(z,\sigma) = C(n) \left(|z|^4 + 16 \sigma^2\right)^{-\frac{n}{2}},
\end{equation}
where $C(n)>0$ is a suitable explicit constant}.

Here, we have indicated with $(z,\sigma)\in \R^{2n+1}$ the real coordinates of a point in $\Hn$. We also note that the normalisation constants in the above expression of $\mathscr E(z,\sigma)$ are not the same as those in \cite{Fo}, where a different group law was adopted. In this paper we always use the group law dictated by the Baker-Campbell-Hausdorff formula, see the opening of Section \ref{S:prelim} below. 

The above theorem has played a pivotal role in the development of analysis in $\Hn$. One of the main reasons lies in the interpretation of this Lie group as the boundary of the Siegel upper half-space in $\mathbb C^{n+1}$, 
see in this connection the seminal work \cite{FScpam}. But  \eqref{T:folland} has also entered in problems with a geometric flavour, such as the theory of conformal and quasiconformal mappings, the CR Yamabe problem and that of the best constants in the Hardy-Littlewood-Sobolev inequality on $\Hn$, see the celebrated works of Kor\'anyi and Reimann \cite{KR}, Jerison and Lee \cite{JL} and Frank and Lieb \cite{FL}. We also mention that Folland's theorem was generalised by Kaplan to all groups of Heisenberg type in \cite[Theorem 2]{Ka} (see also the earlier work \cite[Theorem 1]{KP} where the same result was proved for groups of Iwasawa type).
 One notable aspect of \eqref{T:folland} is the resemblance with the fundamental solution of $-\Delta$ which for $n\ge 3$ is given by $c(n) |x|^{2-n}$. To see this, denote by $Q = 2n+2$ the homogeneous dimension of $\Hn$ attached to the non-isotropic group dilations $\delta_\la(z,\sigma) = (\la z,\la^2 \sigma)$, then we can rewrite $\mathscr E(z,\sigma) = C(n) N(z,\sigma)^{2-Q}$, where $N(z,\sigma) = (|z|^4 + 16 \sigma^2)^{1/4}$ is the so called \emph{gauge function} on $\Hn$ \footnote{Such denomination was introduced by Koranyi and Vagi in \cite{KV}. The subadditivity of the gauge $N(z,\sigma)$ in $\Hn$ or more in general in groups of Heisenberg type was proved in \cite{Cy2}, see also \cite[Section F]{KR}. As a consequence, such particular gauge defines a distance on the group.}.  
Since it appears evident that there is no unique choice of a homogeneous gauge (in principle, all properly normalised functions such as $(|z|^{4k} + \sigma^{2k})^{1/{4k}}$, $k\in \mathbb N$,  might be deemed as reasonable choices), one is left with wondering how does one a priori know that the function $N(z,\sigma)$ is the natural choice to try? 

As a by-product of the results in this note we provide an answer to this question by resorting to the one object which occupies a central position in analysis and geometry: the ubiquitous heat equation. More precisely, suppose one does not know a priori the magic gauge function $N(z,\sigma)$ in $\Hn$, or more in general in a group of Heisenberg type $\bG$. Corollary \ref{C:Htypes} below shows that, by running the heat flow on $\bG$, one is naturally lead into such function. 
As it will be clear from the subsequent discussion, this result can be viewed as the limiting value ($s =1$) where the results for two different families of nonlocal operators merge: the former can be defined in an arbitrary stratified nilpotent Lie group, aka a Carnot group, and can be almost entirely analysed by semigroup methods; the latter has instead its roots in conformal CR geometry. As a consequence, its natural geometric framework is not a general Carnot group, but rather the Heisenberg group $\Hn$ (which is the prototypical CR manifold),  and it has so far been studied by a combination of different ideas which include scattering and non-commutative harmonic analysis. In accordance with the notation adopted by Frank and Lieb in their cited work \cite{FL}, if $\mathscr L$ is a horizontal Laplacian on $\bG$ as in \eqref{L} below, we agree to indicate $\mathscr L^s = (-\mathscr L)^s$, whereas $\mathscr L_s$ will denote the conformal fractional power of the horizontal Laplacian on $\bG$. The ambient for the results in this paper will be that of groups of Heisenberg type (for the relevant notion and the main properties of such groups see Section \ref{S:prelim} below). Henceforth, we will adhere to the convention of using the superscript $s$ for any quantity which has to do with $\mathscr L^s$, whereas we will use the subscript $s$ for anything that has to do with $\mathscr L_s$. Thus, for instance, the fundamental solution of $\mathscr L^s$  is denoted by $\mathscr E^{(s)}(g)$, whereas we use $\mathscr E_{(s)}(g)$ for that of $\mathscr L_s$. 

The main objective of the present work is to show that, notwithstanding their substantial differences, these two classes of nonlocal operators can be treated in a unified way by a systematic use of the heat equation and suitable modifications of the latter. 
To describe our framework consider first a general Carnot group $\bG$ with a fixed horizontal Laplacian $\mathscr L$. Throughout this paper 
we indicate by $P_t u(g) = e^{-t\mathscr L} u(g)= \int_{\bG} p(g,g',t) u(g') dg'$ the heat semigroup constructed by Folland in \cite{Fo75}. We recall that such semigroup is stochastically complete, i.e., $P_t 1 = 1$. The semigroup $P_t$ is all that is needed to study the fractional powers $\mathscr L^s$, for $0<s<1$. We emphasise that we define the action of this nonlocal operator on a function $u\in C^\infty_0(\bG)$ by the well-known formula of Balakrishnan \cite{B},
\begin{equation}\label{bala}
\mathscr L^s u(g) = -\frac{s}{\G(1-s)} \int_0^\infty \frac{1}{t^{1+s}} (P_t u(g) - u(g)) dt.
\end{equation}
With \eqref{bala} in hands, we next consider the Riesz potentials defined  by the formula
\begin{equation}\label{rpotG0}
\mathscr I^{(2s)} u(g) = \frac{1}{\Gamma(s)} \int_0^\infty t^{s -1} P_t u(g) dt,\ \ \ \ \ \ \ \ \ 0<s<1,
\end{equation}
see \cite{Fo75}.
It is easy to prove, see Proposition \ref{P:Isuno} below, that
\begin{equation}\label{caposotto}
\mathscr I^{(2s)} \circ \mathscr L^s = \mathscr L^s \circ \mathscr I^{(2s)} = I.
\end{equation}
A direct important consequence of \eqref{caposotto} is that the kernel
\begin{equation}\label{EsH0}
\mathscr E^{(s)}(g)  \overset{def}{=}  \frac{1}{\Gamma(s)} \int_0^\infty t^{s-1}  p(g,t) dt
\end{equation}
 of the operator $\mathscr I^{(2s)}$ constitutes the fundamental solution of the nonlocal operator $\mathscr L^s$ with pole at the group identity. In Theorem \ref{T:fss} below, we provide an explicit integral expression for such kernel in the setting of groups of Heisenberg type. In discrepancy with the above mentioned Corollary \ref{C:Htypes}, such result shows in particular that in the logarithmic coordinates $g = (z,\sigma)\in \bG$ one has $\mathscr E^{(s)}(z,\sigma) = \Phi(|z|^4,|\sigma|^2)$, but gauge symmetry breaks down when $0<s<1$. However, in the limit as $s\nearrow 1$ the conformal geometry of the group appears, in the sense that with further work it is possible to recover Corollary \ref{C:Htypes} from Theorem \ref{T:fss}.
 
This leads us to introduce the second pseudo-differential operator $\mathscr L_s$. Following the work by Branson, Fontana and Morpurgo \cite[(1.33)]{BFM}, the nonlocal operator $\mathscr L_s$ can be defined in the Heisenberg group $\Hn$, with $T = \p_\sigma$, via the spectral formula 
\begin{equation}\label{bfm}
\mathscr L_s = 2^s |T|^s \frac{\G(-\frac 12\mathscr L |T|^{-1} + \frac{1+s}2)}{\G(-\frac 12 \mathscr L |T|^{-1} + \frac{1-s}2)}.
\end{equation}
Formula \eqref{bfm} is the counterpart of the well-known representation $(-\Delta)^s u = \mathscr F^{-1}(2\pi |\xi|)^{2s} \hat u)$, see \cite[Chapter 5]{St}, except that, as it will soon be clear, now matters are much more involved. More in general, in a group of Heisenberg type $\bG$, with logarithmic coordinates $g= (z,\sigma)\in \bG$, where $\sigma$ is the vertical variable, the pseudo-differential operator $\mathscr L_s$ is defined by the following generalisation of  \eqref{bfm} 
\begin{equation}\label{bfmH}
\mathscr L_s = 2^s (-\Delta_\sigma)^{s/2} \frac{\G(-\frac 12\mathscr L (-\Delta_\sigma)^{-1/2} + \frac{1+s}2)}{\G(-\frac 12 \mathscr L (-\Delta_\sigma)^{-1/2} + \frac{1-s}2)},
\end{equation}
see \cite{RT}. 

From our perspective, the unfavourable aspect of either definitions \eqref{bfm} or \eqref{bfmH} is that if one wants to study the nonlocal operators $\mathscr L_s$ starting from them, then one is immediately led into the fairly elaborate computational aspects connected with non-commutative Fourier analysis on the group $\bG$, thus losing sight of the remarkable flexibility of the heat equation offered by \eqref{bala}. Since the present work is about the heat flow, instead of \eqref{bfmH} we will henceforth adopt the following definition which, at least formally, seems identical to \eqref{bala},
\begin{equation}\label{RTs}
\mathscr L_s u(g) = - \frac{s}{\G(1-s)} \int_0^\infty \frac{1}{t^{1+s}} \left[P_{(-s),t} u(g) - u(g)\right] dt,
\end{equation}
where $0<s<1$, and $u\in C^\infty_0(\bG)$. In \eqref{RTs} we have indicated with $P_{(-s),t}$ a linear operator on $\Lp(\bG)$ that is associated with a modified heat equation and whose origin will be explained in detail in Section \ref{S:secdescr}, but see also the discussion leading to \eqref{conn} below.  The equivalence between formulas  \eqref{bfmH} and \eqref{RTs}   was established by Roncal and Thangavelu in Proposition 4.1 of their remarkable paper \cite{RTaim} on optimal Hardy inequalities, see also the companion work \cite{RT} in which they generalised their results to groups of Heisenberg type. 

Having defined $\mathscr L_s$ via \eqref{RTs}, for $0<s\le 1$ we now introduce two modified heat flows on $\bG$. We consider the kernel
\begin{equation}\label{poisson}
\mathscr K_{(s)}((z,\sigma),t) = \frac{2^k}{\left(4\pi t\right)^{\frac m2+k}}\int_{\R^k} e^{-\frac it \langle \sigma,\la\rangle} \left(\frac{|\la|}{\sinh |\la|}\right)^{\frac m2+1-s}  e^{- \frac{|z|^2}{4t} \frac{|\la|}{\tanh |\la|}}d\la,
\end{equation}
and we denote by $\mathscr K_{(-s)}((z,\sigma),t)$ the function obtained by changing $s$ into $-s$ in \eqref{poisson}. We note that, in the local case $s = 1$, the kernel $\mathscr K_{(s)}$  coincides with the Gaveau-Hulanicki-Cygan heat kernel $p((z,\sigma),t)$, see definition \eqref{pipposemprepiupiccoloHoo} below. If with a slight abuse of notation we let 
\[
\mathscr K_{(\pm s)}(g,g',t)= \mathscr K_{(\pm s)}(g^{-1}\circ g',t),
\]
then we consider the two linear operators on $\Lp(\bG)$ defined by the formula
\begin{equation}\label{Pst}
P_{(\pm s),t} u(g) = \int_{\bG} \mathscr K_{(\pm s)}(g,g',t) u(g') dg'.
\end{equation}
We note explicitly that $P_{(-s),t}$ is the Roncal-Thangavelu operator  that appears in \eqref{RTs} (as a help to the reader we mention that in \cite[Formula (2.18)]{RTaim} they denote by $\mathcal K^s_t$ what we indicate with $\mathscr K_{(-s),t}$). 
 
At this point, using the operator $P_{(s),t}$ we introduce the  counterpart of the Riesz operator \eqref{rpotG0} above
\begin{equation}\label{conformalriesz}
\mathscr I_{(2s)} u(g) = \frac{1}{\G(s)} \int_0^\infty t^{s-1} P_{(s),t} u(g) dt.
\end{equation}
We stress that, unlike \eqref{bala} and \eqref{rpotG0}, which are both defined using the same semigroup $P_t$, in \eqref{RTs} and \eqref{conformalriesz} two different modified heat operators appear. We are now in a position to state our first main result which represents the conformal counterpart of \eqref{caposotto}. 
 
\begin{theorem}\label{T:I2sdue0}
For every $0<s<1$ and $u\in C_0^\infty(\bG)$ one has
$$\left( \mathscr  I_{(2s)} \circ \mathscr L_s \right) u= \left( \mathscr L_s \circ \mathscr I_{(2s)} \right) u= u.$$
\end{theorem}

We emphasise that our proof of Theorem \ref{T:I2sdue0}, which is given in Section \ref{S:inv}, is based on some lemmas of independent interest which are inspired to semigroup methods. In particular, in Lemma \ref{kernallazzamento} we establish a key representation formula for the group convolution of the intertwined kernels $\mathscr K_{(s)}(\cdot,t)$ and $\mathscr K_{(-s)}(\cdot,\tau)$. We also mention Lemma \ref{lemmaMonster}, which establishes a remarkable cancellation property. 
In the following theorem, which is our second main result, we compute the kernel of the operator \eqref{conformalriesz} explicitly. 

\begin{theorem}\label{T:poisson}
Let $\bG$ be a group of Heisenberg type. The following statements hold:
\begin{itemize}
\item[(i)] With $\mathscr K_{(s)}$ defined by \eqref{poisson}, for any $0<s\le 1$ one has 
\begin{equation}\label{Esottos}
\mathscr E_{(s)}(z,\sigma) \overset{def}{=} \frac{1}{\G(s)} \int_0^\infty t^{s-1} \mathscr K_{(s)}((z,\sigma),t) dt  = \frac{C_{(s)}(m,k)}{N(z,\sigma)^{Q-2s}},
\end{equation}
where $Q = m + 2k$ is the homogeneous dimension of $\bG$, $N(z,\sigma) = (|z|^4 + 16 |\sigma|^2)^{1/4}$ is the natural gauge, and we have let
\begin{equation}\label{C}
C_{(s)}(m,k) = \frac{2^{\frac m2 + 2k-3s-1} \G(\frac 12(\frac m2+1-s)) \G(\frac 12(\frac m2 + k -s))}{\pi^{\frac{m+k+1}2} \G(s)}.
\end{equation}
\item[(ii)] The distribution $\mathscr E_{(s)}\in C^\infty(\bG\setminus\{e\})\cap L^1_{loc}(\bG)$, and it provides a fundamental solution of $\mathscr L_s$ with pole at the group identity $e\in \bG$ and vanishing at infinity.
\end{itemize}
\end{theorem}

The reader should note the notable similarity between \eqref{Esottos} and the well-known result, see e.g. \cite[Theorem 8.4]{Gft}, stating that the function
$
E_s(x) = \frac{c(n,s)}{|x|^{n-2s}}
$
is the fundamental solution with pole at $x=0$ of the nonlocal operator $(-\Delta)^s$.
 Theorem \ref{T:poisson} provides a first example of the  remarkable connection between the heat equation in a group of Heisenberg type and its conformal geometry. 
We mention that  \eqref{Esottos} was first obtained with a completely different approach, based on the Fourier analysis on groups, by Roncal and Thangavelu in \cite[Section 3]{RTaim}. For instance, in order to invert the operator $\mathscr L_s$ they rely on the deep formula of Cowling-Haagerup \cite{CH} which gives the group Fourier transform of the kernels  $\mathscr E_{(-s)}$ obtained by changing $s$ into $-s$ in \eqref{Esottos}. Instead, we deduce the invertibility of $\mathscr L_s$ directly from our Theorem \ref{T:I2sdue0}, and then prove \eqref{Esottos}, \eqref{C} independently. Another difference  is that we treat all groups of Heisenberg type at once, whereas in \cite{RTaim} the authors first establish the Heisenberg group case, and then in \cite{RT} they use partial Radon transform and several facts from Lie theory to ultimately reduce matters to the case of $\Hn$.

Concerning the question in the opening of this paper, at this moment it seems appropriate to state separately  the special case $s =1$ of Theorem \ref{T:poisson}. As we have already mentioned, the next result provides a heat equation proof of Folland's formula \eqref{T:folland} and its generalisation in the cited paper \cite{Ka} by Kaplan. 

\begin{corollary}[Discovering the gauge from the heat]\label{C:Htypes}
Let $\bG$ be a group of Heisenberg type. Then, the fundamental solution of $-\mathscr L$ is given by
\[
\int_0^\infty p((z,\sigma),t) dt  = \frac{2^{\frac m2 + 2k-2}\G(\frac m4) \G(\frac 12(\frac m2 + k -1))}{\pi^{\frac{m+k+1}2}} \left(|z|^4 + 16 |\sigma|^2\right)^{-\frac{1}{2}(\frac m2 + k-1)}.
\] 
\end{corollary}

As we have said, the present work is purely based on the analysis of various heat kernels using techniques inspired by pde's and semigroup theory. The common starting point of our analysis are the parabolic extension problems associated with both nonlocal operators $\mathscr L^s$ and $\mathscr L_s$. While we defer a detailed discussion to Section \ref{S:secdescr}, to provide the reader with some perspective here we confine ourselves to recall such problems. Given a function $u\in C^\infty_0(\bG\times \R_t)$, the extension problem for  $(\p_t - \mathscr L)^s$ consists in finding $U\in C^\infty(\bG\times \R_t\times \R^+_y)$ such that
\begin{equation}\label{extHs0}
\begin{cases}
\mathfrak P^{(s)} U \overset{def}{=} \frac{\p^2 U}{\p y^2}  + \frac{1-2s}y \frac{\p U}{\p y} + \mathscr L U - \frac{\p U}{\p t} = 0,
\\ 
U(g,t,0) = u(g,t).
\end{cases}
\end{equation}
Here, we have let $g = (z,\sigma)\in \bG$, and we have denoted by $y>0$ the extension variable.
The conformal counterpart of \eqref{extHs0} is formulated in a similar way, but the problem is substantially different.
Given a function $u\in C^\infty_0(\bG\times \R_t)$, find a function $U\in C^\infty(\bG\times  \R_t \times \R^+_y)$ such that
\begin{equation}\label{parext0}
\begin{cases}
\mathfrak P_{(s)} U \overset{def}{=} \frac{\p^2 U}{\p y^2}  + \frac{1-2s}y \frac{\p U}{\p y} + \frac{y^2}4 \Delta_\sigma U+ \mathscr L U - \frac{\p U}{\p t} = 0,\ \ \ \ \ \text{in}\ \bG\times  \R_t \times \R^+_y,
\\
U(g,t,0) = u(g,t).
\end{cases}
\end{equation}
The presence of the differential operator $\frac{y^2}4 \Delta_\sigma$ is what makes \eqref{parext0} so diverse from \eqref{extHs0}, but it is also what gives a geometric meaning to \eqref{parext0}. To justify this statement we recall that the fundamental solution of the operator $\mathfrak P^{(s)}$ in \eqref{extHs0} is given by
$$q^{(s)}(g,g',t,y) \overset{def}{=} g^{(s)}(y,t)  p(g,g',t),$$
where $p(g,g',t)$ is the heat kernel in $\bG$, and we have let $g^{(s)}(y,t) =(4\pi t)^{-(1-s)} e^{-\frac{y^2}{4t}}$ denote the heat kernel in the space with fractal dimension $\R_y^{2(1-s)}\times \R^+_t$. On the other hand, the fundamental solution of the operator $\mathfrak P_{(s)}$ in \eqref{parext0} is given by
the function 
\begin{align}\label{parBfs0}
 q_{(s)}((z,\sigma),t,y) & =  \frac{2^k}{(4\pi t)^{\frac{m}2 +k +1-s}} \int_{\R^k} e^{- \frac it \langle \sigma,\la\rangle}   \left(\frac{|\la|}{\sinh |\la|}\right)^{\frac m2+1-s} e^{-\frac{|z|^2 +y^2}{4t}\frac{|\la|}{\tanh |\la|}} d\la.
\end{align}
The origin of this function is in the fact that $\mathfrak P_{(s)}$ is to be viewed as a parabolic Baouendi-Grushin operator (see \eqref{salahpar}) in the  space with fractal dimension $\R^{m+2(1-s)}\times \R^k\times (0,\infty)$ and whose fundamental solution can be explicitly computed, see Proposition \ref{P:BGpar}. Once this is recognised, then the connection  between the kernel \eqref{poisson} of the operators $P_{(s),t}$ in the conformal Riesz operator \eqref{conformalriesz} and the function \eqref{parBfs0} is given by the formula
\begin{equation}\label{conn}
\mathscr K_{(s)}((z,\sigma),t) = (4\pi t)^{1-s} q_{(s)}((z,\sigma),t,0).
\end{equation}
In other words, $\mathscr K_{(s)}$ is an appropriately rescaled restriction to the thin space $y = 0$ of the Baouendi-Grushin kernel $q_{(s)}$ in \eqref{parBfs0}. 

To unravel the conformal geometry in \eqref{parBfs0} we now note that, from general principles, we know that
\begin{equation}\label{frake}
\mathfrak e_{(s)}((z,\sigma,y) \overset{def}{=} \int_0^\infty q_{(s)}((z,\sigma),t,y) dt
\end{equation}
is a fundamental solution with pole at the origin of the time-independent part of $\mathfrak P_{(s)}$, i.e., the conformal extension operator
\begin{equation}\label{tiepa}
\mathfrak L_{(s)} = \frac{\p^2}{\p y^2} + \frac{1-2s}{y} \frac{\p }{\p y} + \frac{y^2}4 \Delta_\sigma + \mathscr L.
\end{equation}
This observation leads us to state the following result.

\begin{theorem}\label{T:thick}
Let $0<s\le 1$. In any group of Heisenberg type $\bG$, the distribution in the thick space $\bG\times \R^+_y$ defined by \eqref{frake} is given by
\begin{equation}\label{qmasomeno}
\mathfrak e_{(s)}((z,\sigma),y) = \frac{\G(s)}{(4\pi)^{1-s}} C_{(s)}(m,k)\ ((|z|^2+y^2)^2+16|\sigma|^2)^{-\frac 12(\frac m2 + k - s)},
\end{equation}
where $C_{(s)}(m,k)$ is the constant in \eqref{C}.
An equation similar to \eqref{qmasomeno} holds if we replace $s$ with $-s$, provided that $\G(s)$ is replaced by $|\G(-s)|$.
\end{theorem}

\begin{remark}\label{R:poisson} 
The reader should note that, changing $s$ into $-s$ in the constant $C_{(s)}(m,k)$, we  obtain from \eqref{qmasomeno} exactly the same number in formula (1.9) in \cite[Theorem 1.2]{RT}. This can be easily recognised by using in their formula Legendre's duplication property for the gamma function recalled in \eqref{prod} below. 
\end{remark}

In closing, we mention that in the Heisenberg group $\Hn$ the time-independent extension problem for the operator \eqref{tiepa} (see \eqref{Ext} below) was first introduced and solved by Frank, Gonzalez, Monticelli and Tan in \cite{FGMT} using harmonic analysis and scattering theory. The same extension problem was also studied by M\"ollers, Orsted and Zhang in \cite{MOZ}. Exploiting the conformal invariance of the differential operators they used unitary representation theory of reductive Lie groups to solve the problem.
A point of view different from these authors was taken up by Roncal and Thangavelu in their cited works \cite{RTaim, RT} on optimal Hardy inequalities. In these papers the authors use the parabolic extension problem \eqref{parext0} and Fourier analysis on groups to establish an explicit Poisson representation formula for the solution of the time-independent problem \eqref{Ext}. One of the purposes of the present work is to further develop the point of view in \cite{RTaim, RT} from a  different semigroup perspective with the objective to unify the treatment of the very diverse nonlocal operators $\mathscr L^s$ and $\mathscr L_s$.

A brief discussion about the organisation of our work seems in order. In Section \ref{S:prelim} we collect various known facts that will be needed in the main body of the paper. 
In Section \ref{S:secdescr} we describe in detail the evolutive extension problems for $\mathscr L^s$ and $\mathscr L_s$ from a unifying perspective. It is there that we introduce the fundamental solution $q_{(s)}$ in \eqref{parBfs0}, its companion $q_{(-s)}$, and the intertwined operators $P_{(\pm s),t}$ in \eqref{Pst}. Section \ref{S:inv} is entirely devoted to proving Theorem \ref{T:I2sdue0}. In Section \ref{S:one} we prove Theorems \ref{T:poisson} and \ref{T:thick}.

\medskip

\noindent \textbf{Acknowledgment:} We are grateful to M. Cowling, G. Folland, A. Kor\'anyi, C. Morpurgo and S. Thangavelu for providing us with some interesting historical overviews during the preparation of this work. Special thanks go to S. Thangavelu for his insightful discussions of his joint papers with L. Roncal.\\
We also thank the anonymous referee for his/her very careful reading of the original manuscript and constructive comments that have contributed to improve the presentation of the paper.


\section{Preliminaries}\label{S:prelim}

In this section we gather some preliminary material which will be needed in the rest of the paper. We begin with recalling a beautiful formula from classical analysis, namely the Fourier transform of the measure carried by the unit sphere. In what follows we denote by $d\omega$ the $(k-1)$-dimensional surface measure on $\mathbb S^{k-1}$, and by $J_\nu$ the Bessel function of the first kind and order $\nu\in \mathbb C$. For $\Re\nu>-\frac12$ the Poisson representation of such function is  
\[
J_\nu(z)
=\frac{1}{\G(\frac12)\G(\nu+\frac{1}{2})}\left(\frac{z}2\right)^\nu\int^1_{-1}
e^{izt}(1-t^2)^{\frac{2\nu-1}2}dt,
\]
where $\G(x)$ denotes the Euler gamma function, see \cite{Wa}. The following classical formula plays an important role in harmonic analysis and pde's, especially in connection with the restriction problem for the Fourier transform and the Cauchy problem for the wave equation. For its proof see p. 154 in \cite{SW}.

\begin{proposition}\label{P:FTspheren}
Assume that $k\ge 2$. Then, for any $\xi\in \R^k$ one has
\[
\int_{\mathbb S^{k-1}} e^{-2\pi i \langle \xi,\omega\rangle} d\omega = 2\pi |\xi|^{-\frac{k}2+ 1}
J_{\frac{k}2-1}(2\pi |\xi|).
\]
\end{proposition}

Another family of special functions that will be needed in this paper are the Gauss hypergeometric functions. We recall the definition of the Pochammer symbols
\[
\alpha_0 = 1,\ \ \ \ \alpha_k \overset{def}{=} \frac{\G(\alpha + k)}{\G(\alpha)} = \alpha(\alpha+1)...(\alpha + k -1),\ \ \ \ \ \ \ \ \ \ k\in \mathbb N. 
\] 
Notice that, since the gamma function has a pole in $z=0$, we have $0_k = 1$ if $k = 0$, $0_k = 0$ for $k\ge 1$.

\begin{definition}\label{D:hyper}
Let $p, q\in \mathbb N \cup\{0\}$ be such that $p\le q+1$, and let $\alpha_1,...,\alpha_p$ and $\beta_1,...,\beta_q$ be given parameters such that $-\beta_j\not\in \mathbb N \cup\{0\}$ for $j=1,...,q$. Given a number $z\in \mathbb C$, the power series
\[
_p F_q(\alpha_1,...,\alpha_p;\beta_1,...,\beta_q;z) = \sum_{k=0}^\infty \frac{(\alpha_1)_k . . . (\alpha_p)_k}{(\beta_1)_k . . . (\beta_q)_k} \frac{z^k}{k!}
\]
is called the \emph{generalized hypergeometric function}. When $p = 2$ and $q=1$, then the function 
\begin{equation}\label{F}
_2 F_1(\alpha_1,\alpha_2;\beta_1;z) = \frac{\G(\beta_1)}{\G(\alpha_1)\G(\alpha_2)} \sum_{k=0}^\infty \frac{\G(k + \alpha_1) \G(k+\alpha_2)}{\G(k+\beta_1) k!} z^k,
\end{equation}
is the \emph{Gauss' hypergeometric function}, and it is usually denoted by $F(\alpha_1,\alpha_2;\beta_1;z)$.
\end{definition}
One special case that we will need is the following simple, yet important  fact, which can be directly verified from \eqref{F} (see also formula (4) on p. 101 in \cite{E})
\begin{equation}\label{fs6}
F(\alpha,\beta;\beta;-a) =\ _1F_0(\alpha;-a) = (1+a)^{-\alpha}.
\end{equation}

Next, we recall the following Kummer's relation concerning how the hypergeometric function $F$ changes under linear transformations (see  formula (3) on p. 105 in \cite{E}, or also (9.5.1) on p. 247 in \cite{Le}),
\begin{equation}\label{hyperG}
F(\alpha,\beta;\gamma;u) = (1-u)^{-\alpha} F\left(\alpha,\gamma - \beta;\gamma;\frac{u}{u-1}\right),\ \ \ \ \ \ u\not=1, \ |\arg(1-u)|<\pi. 
\end{equation}
The following classical formula due to Gegenbauer can be found in (3) on p. 385 in \cite{Wa} (see also 6.621.1 on p. 711 in \cite{GR}) 
\begin{equation}\label{ohoh}
\int_0^\infty t^{\mu-1} e^{-\alpha t} J_\nu(\beta t) dt = \frac{2^{-\nu} \beta^\nu \G(\nu+\mu)}{\G(\nu+1) (\alpha^2 + \beta^2)^{\frac{\nu+\mu}2}} F\left(\frac{\nu+\mu}{2},\frac{1-\mu+\nu}{2};\nu+1;\frac{\beta^2}{\alpha^2 + \beta^2}\right), 
\end{equation}
provided that
\[
\Re(\nu+\mu)>0,\ \Re(\alpha+i\beta)>0,\ \Re(\alpha-i\beta)>0.
\]
Finally, we need the following formula due to H. Bateman (see formula (2) on p. 78 in \cite{E}, and also Problem 6. on p. 277 in \cite{Le})
\begin{equation}\label{hyperGint}
\int_0^1 y^{c-1} (1-y)^{\gamma-c-1} F(\alpha,\beta;c;a y) dy = \frac{\G(c)\G(\gamma - c)}{\G(\gamma)} F(\alpha,\beta;\gamma;a),
\end{equation}
provided that $\Re \gamma >\Re c>0$, $a\not=1$ and $|\arg(1-a)|<\pi$. The dividing line between Theorems \ref{T:fss} and \ref{T:poisson} will be precisely formula \eqref{hyperGint}. 
\subsection{Groups of Heisenberg type}\label{SS:Htype}

We recall that a Carnot group of step $r = 2$ is a simply-connected Lie group $\bG$ whose Lie algebra admits a stratification $\bg = V_1 \oplus V_2$, with $[V_1,V_1] = V_2$, $[V_1,V_2] = \{0\}$. We assume that $\bg$ is endowed with an inner product $\langle \cdot,\cdot\rangle$ and induced norm $|\cdot|$, and we let $m = \operatorname{dim}(V_1)$, $k = \operatorname{dim}(V_2)$.  We fix orthonormal basis $\{e_1,...,e_{m}\}$ and $\{\ve_1,...,\ve_k\}$ for $V_1$ and $V_2$ respectively, and for points $z\in V_1$ and $\sigma\in V_2$ we will use either one of the representations $z = \sum_{j=1}^{m} z_j e_j$, $\sigma = \sum_{\ell=1}^k \sigma_\ell \ve_\ell$, or also $z = (z_1,...,z_{m})$, $\sigma = (\sigma_1,...,\sigma_k)$. Accordingly, whenever convenient we will identify the point $g = \exp(z+\sigma)\in \bG$ with its logarithmic coordinates $(z,\sigma)$. The Kaplan mapping $J: V_2 \to \operatorname{End}(V_1)$ is defined by 
\begin{equation}\label{kap}
\langle J(\sigma)z,\zeta\rangle = \langle [z,\zeta],\sigma\rangle = - \langle J(\sigma)\zeta,z\rangle.
\end{equation}
Clearly, $J(\sigma)^\star = - J(\sigma)$, and one has $\langle J(\sigma)z,z\rangle = 0$. By \eqref{kap} and the Baker-Campbell-Hausdorff formula, see p. 12 of \cite{CG}, 
$$\exp(z+\sigma)  \exp(\zeta+\tau) = \exp{\left(z + \zeta + \sigma +\tau + \frac{1}{2}
[z,\zeta]\right)},$$
we obtain the non-Abelian multiplication in $\bG$
\begin{equation}\label{grouplaw}
g\circ g' = \big(z +\zeta,\sigma + \tau + \frac 12 \sum_{\ell=1}^{k} \langle J(\ve_\ell)z,\zeta\rangle\ve_\ell\big).
\end{equation}
The horizontal Laplacian associated with the basis $\{e_1,...,e_{m}\}$ of the layer $V_1$ is the second-order partial differential operator on $\bG$ defined by 
\begin{equation}\label{L}
\mathscr L f =  \sum_{j=1}^{m} X_j^2 f,
\end{equation}
where $X_1,...,X_{m}$ are the left-invariant vector fields in $\bG$ given by the Lie rule $$X_j u(g) = \frac{d}{ds} u(g \circ \exp s e_j)\big|_{s=0}.$$
The operator \eqref{L} fails to be elliptic at every point of $\bG$, but it is hypoelliptic thanks to the grading assumption on the Lie algebra and to H\"ormander's theorem in \cite{Ho}. The prototype par excellence of a Carnot group of step two is of course the Heisenberg group of real dimension $2n+1$, see e.g. \cite{Fo},  \cite{FScpam} and \cite{Ko}. More in general, a Carnot group of step two $\bG$ is said of Heisenberg type if for every $\la\in V_2$ one has
$$J(\la)^2 = - |\la|^2 I_m,$$
see \cite{Ka}. There is in nature a plentiful supply of such groups. For instance, the nilpotent component in the Iwasawa decomposition of a simple group of rank one is a group of Heisenberg type, see  \cite{KP}, \cite{Co} and \cite{CDKR}. All the results in this paper are valid for such class of Lie groups. 

It is well-known that, when $\bG$ is of Heisenberg type, then in the logarithmic coordinates $(z,\sigma)\in \bG$ one has
\begin{equation}\label{L2}
\mathscr L = \Delta_z + \frac{|z|^2}{4} \Delta_\sigma + \sum_{\ell = 1}^k \Theta_\ell \p_{\sigma_\ell},
\end{equation}
where $\Theta_\ell = \sum_{s=1}^m \langle J(\ve_\ell)z,e_s\rangle \p_{z_s}$, see e.g. \cite[Section 2.5]{Gparis}. When $U(z,\sigma) = u(|z|^2,\sigma)$, then for every $\ell =1,...,k$ we have
\begin{equation}\label{theta}
\Theta_\ell U = 2 \partial_1 u(|z|^2,\sigma) \sum_{s=1}^m \langle J(\ve_\ell)z,e_s\rangle z_s = 2 \partial_1 u(|z|^2,\sigma) \langle J(\ve_\ell)z,z\rangle = 0,
\end{equation}
and one has from \eqref{L2}
\begin{equation}\label{salah}
\mathscr L U = \Delta_z U + \frac{|z|^2}{4} \Delta_\sigma U.
\end{equation}


\subsection{The heat kernel}\label{SS:heat}

We close this section by recalling a formula that is at the core of the present work. In a group of Heisenberg type $\bG$ consider the heat operator $\p_t - \mathscr L$, where $\mathscr L$ is as in \eqref{L}. The fundamental solution of this operator with pole at the group identity is given by 
\begin{align}\label{pipposemprepiupiccoloHoo}
& p(z,\sigma,t) = \frac{2^k}{\left(4\pi t\right)^{\frac m2+k}}\int_{\R^k} e^{-\frac it \langle \sigma,\la\rangle}  \left(\frac{|\la|}{\sinh |\la|}\right)^{\frac m2} e^{- \frac{|z|^2}{4t} \frac{|\la|}{\tanh |\la|}}d\la,
\end{align}
where we have identified a point $g\in \bG$ with its logarithmic coordinates $(z,\sigma)\in \R^m\times\R^k$. By left-invariance, the function $p(g^{-1} \circ g',t)$ provides a fundamental solution of the heat equation with pole at any other $g'\in \bG$. We note in passing that, because of the complex structure induced by the Kaplan map $J$ in \eqref{kap}, in any group of Heisenberg type the dimension of the first layer must be even, i.e., $m=2n$ for some $n\in \mathbb N$, although we will never use this fact.
Formula \eqref{pipposemprepiupiccoloHoo} was found independently by Hulanicki \cite{Hu} and Gaveau \cite{Gav} in the Heisenberg group $\Hn$. Subsequently, their result was generalised by Cygan \cite{Cy} to all groups of step two, see also our recent work \cite{GTstep2} for a new approach to Cygan's result and for a detailed account of the literature. 

While in the analysis of $\mathscr L^s$ the only character is the heat kernel \eqref{pipposemprepiupiccoloHoo}, to study the nonlocal operator $\mathscr L_s$ one needs to carefully exploit the intertwining properties of the two operators $P_{(\pm s)}$. The next two sections are devoted to such task.


\section{Extension problems and intertwining heat kernels}\label{S:secdescr}

In this section we describe in detail the extension problems for $\mathscr L^s$ and $\mathscr L_s$. In  geometry the importance of extension procedures in the study of conformal invariants was highlighted in the celebrated works \cite{FG, GZ}. In our situation we will see that the relevant evolution pde in a higher-dimensional space leads to consider in a natural fashion the modified heat kernels introduced in \eqref{poisson}. In the classical setting we recall the (not so well-known) pioneering paper by F. Jones \cite{Jo} in which this author first solved the extension problem for the fractional heat equation $(\p_t - \Delta)^{1/2}$ in $\Rn$ by constructing an explicit Poisson kernel. We also mention the landmark work by Caffarelli and Silvestre \cite{CS}, in which they solved the extension problem for $(-\Delta)^s$ for arbitrary fractional powers $0<s<1$, and they established an alternative interpretation of this pseudo-differential operator as a weighted Dirichlet-to-Neumann map of a certain degenerate operator in one dimension up. Their paper has been, and continues to be, a rich source of development both in analysis and geometry.

\subsection{The model non-geometric extension problem for $\mathscr L^s$}

Similarly to what was done in  \cite{CS}, an alternative interpretation of the nonlocal operator \eqref{bala} is via the so-called extension problem:
given a function $f\in C^\infty_0(\bG)$, find $F\in C^\infty(\bG\times \R^+_y)$ such that
\begin{equation}\label{extLsopras}
\begin{cases}
\frac{\p^2 F}{\p y^2}  + \frac{1-2s}y \frac{\p F}{\p y} + \mathscr L F  = 0,
\\ 
F(g,0) = f(g).
\end{cases}
\end{equation}
We recall that in the framework of Carnot groups the problem \eqref{extLsopras} was studied by Ferrari and Franchi in \cite{FF}.
Since the present work is about the heat equation,  inspired by Bochner's ideas, instead of \eqref{extLsopras} we will consider the extension problem for the fractional powers $(\p_t - \mathscr L)^s$ and then use subordination to recover the solution of \eqref{extLsopras}. This nonlocal evolution operator is defined by a formula similar to \eqref{bala}, but in which the semigroup $P_t = e^{-t\mathscr L}$ is replaced by the evolutive heat semigroup defined by 
\[
P^{\mathscr H}_\tau u(g,t) = \int_{\bG} p(g,g',\tau) u(g',t-\tau)dg'.
\]
In this perspective the reader should see \cite{Gejde} for a treatment of the case when $\mathscr L$ is a general H\"ormander sub-Laplacian, and also \cite{GTiumj}, where we studied the extension problem for the fractional powers of evolution hypoelliptic operators of Fokker-Planck-Kolmogorov type. 
  Given a function $u\in C^\infty_0(\bG\times \R_t)$, the extension problem for  $(\p_t - \mathscr L)^s$ consists in finding $U\in C^\infty(\bG\times \R_t\times \R^+_y)$ such that
\begin{equation}\label{extHs}
\begin{cases}
\mathfrak P^{(s)} U = \frac{\p^2 U}{\p y^2}  + \frac{1-2s}y \frac{\p U}{\p y} + \mathscr L U - \frac{\p U}{\p t} = 0,
\\ 
U(g,t,0) = u(g,t).
\end{cases}
\end{equation}
We note that, by general principles, the function
\begin{equation}\label{qsopras}
q^{(s)}(g,g',t,y) \overset{def}{=} g^{(s)}(y,t)  p(g,g',t)
\end{equation}
is the fundamental solution (with pole at the point $(g',0,0)$ in the thick space $\bG\times \R_t\times \R_y$) of the extension operator $\mathfrak P^{(s)}$. Here, we have let $g^{(s)}(y,t) \overset{def}{=} (4\pi t)^{-(1-s)} e^{-\frac{y^2}{4t}}$ denote the heat kernel in the space with fractal dimension $\R_y^{2(1-s)}\times \R^+_t$. 
The Poisson kernel for the problem \eqref{extHs} is given by 
\begin{equation}\label{Ksopras}
\mathscr P^{(s)}(g,g',t,y) = \frac{4 \pi^{1+s}}{\G(1-s)} y^{2s} g^{(-s)}(y,t)  p(g,g',t),
\end{equation} 
where we have indicated with $g^{(-s)}(y,t) = (4\pi t)^{-(1+s)} e^{-\frac{y^2}{4t}}$ the heat kernel in the space with fractal dimension $\R_y^{2(1+s)}\times \R^+_t$. One should notice that the stochastic completeness of $P_t$, i.e. $P_t 1 = 1$, and a simple computation imply that for every $g\in \bG$ and $y>0$ one has
\begin{equation}\label{pescionone}
\int_0^\infty \int_{\bG} \mathscr P^{(s)}(g,g',t,y) dg' dt = 1.
\end{equation}
With such Poisson kernel in hands, the solution of the extension problem \eqref{extHs} is given by
\begin{align}\label{rem}
U^{(s)}(g,t;y) & = \int_0^\infty \int_{\bG} \mathscr P^{(s)}(g,g',\tau,y) u(g',t-\tau) dg' d\tau,
\end{align}
see \cite{Gejde}. The basic property of the function defined by \eqref{rem} is represented by the following Dirichlet-to-Neumann relation
\begin{equation}\label{DN}
- \frac{2^{2s-1} \G(1-s)}{\G(1+s)} \underset{y\to 0^+}{\lim} y^{1-2s} \frac{\p U^{(s)}}{\p y}(g,t;y) = (\p_t - \mathscr L)^s u(g,t).
\end{equation}
The limit in \eqref{DN} is not only pointwise, but it also holds  in $L^p(\bG\times \R)$ for any $1\le p \le \infty$.

Returning to the Poisson kernel \eqref{Ksopras}, we remark at this point that the function
\begin{equation}\label{sopramenos}
q^{(-s)}(g,g',t,y) \overset{def}{=} g^{(-s)}(y,t)  p(g,g',t),
\end{equation}
is the fundamental solution (with pole at the point $(g',0,0)$ in the thick space $\bG\times \R_t\times \R_y$) of the parabolic differential operator
\begin{equation}\label{Psopras}
\mathfrak P^{(-s)}  = \frac{\p^2}{\p y^2}  + \frac{1+2s}y \frac{\p}{\p y} + \mathscr L  - \frac{\p }{\p t}. 
\end{equation} 
The reader should note that, like the functions $q^{(s)}$ and $q^{(-s)}$, the operator \eqref{Psopras} is obtained by changing $s$ into $-s$ in the definition of $\mathfrak P^{(s)}$ in \eqref{extHs}. The link between the differential operators $\mathfrak P^{(s)}$ and $\mathfrak P^{(-s)}$ is given by the well-known fact, see for instance the 1965 work of Muckenhoupt and Stein \cite{MS}, that one has the following intertwining relation between the two Bessel equations
\[
\frac{\p^2 (y^{2s} u)}{\p y^2} + \frac{1-2s}{y} \frac{\p (y^{2s} u)}{\p y} = y^{2s} \left(\frac{\p^2 u}{\p y^2} + \frac{1+2s}{y} \frac{\p u}{\p y}\right).
\]
As a consequence, we have
\begin{equation}\label{PPsopras}
\mathfrak P^{(s)}\left(y^{2s} q^{(-s)}\right) = y^{2s} \mathfrak P^{(-s)}\ q^{(-s)} = 0,\ \ \ \ \text{and}\ \ \ \ \mathfrak P^{(-s)}\left(y^{-2s} q^{(s)}\right) = y^{-2s} \mathfrak P^{(s)}\ q^{(s)} = 0.
\end{equation}
We note that the former equation in \eqref{PPsopras} shows, in particular, that for every fixed $g'\in \bG$, the Poisson kernel \eqref{Ksopras} solves
\[
\mathfrak P^{(s)} \mathscr P^{(s)}(\cdot,g',\cdot,\cdot) = 0
\]
in the thick space $\bG\times \R^+_t \times \R^+_y$. Before proceeding, we note that from the expressions of $q^{(\pm s)}$ we have the following obvious, yet important fact,
\begin{equation}\label{masomenosopras}
(4\pi t)^{1-s} q^{(s)}(g,g',t,0) = (4\pi t)^{1+s} q^{(-s)}(g,g',t,0) = p(g,g',t),
\end{equation}
the heat kernel in $\bG\times \R$. We also note that, as a basic consequence of Bochner's principle of subordination, if we integrate in time the function in \eqref{Ksopras} we obtain the corresponding Poisson kernel for the extension problem for $\mathscr L^s$ \begin{equation}\label{Psenzatempo}
\mathscr Q^{(s)}(g,g',y) = \int_0^\infty \mathscr P^{(s)}(g,g',t,y) dt= \frac{4 \pi^{1+s}}{\G(s)} y^{2s} \int_0^\infty g^{(-s)}(y,t)  p(g,g',t) dt.
\end{equation} 
One should observe that \eqref{pescionone} implies that
for any $g\in \bG$ and $y>0$ one has
\begin{equation}\label{pescetto}
\int_{\bG} \mathscr Q^{(s)}(g,g',y) dg' = 1. 
\end{equation}
It was shown in \cite{Gejde} that, if one defines
\[
F^{(s)}(g,y) = \int_{\bG} \mathscr Q^{(s)}(g,g',y) f(g') dg' = \int_0^\infty \int_{\bG} \mathscr P^{(s)}(g,g',t,y)f(g') dg' dt,
\]
then this function solves the problem \eqref{extLsopras}, and moreover it satisfies the Dirichlet-to-Neumann condition
\begin{equation}\label{DNell}
- \frac{2^{2s-1} \G(1-s)}{\G(1+s)} \underset{y\to 0^+}{\lim} y^{1-2s} \frac{\p F^{(s)}}{\p y}(g,y) = (-\mathscr L)^s f(g).
\end{equation}


\subsection{$\mathscr L_s$ and the conformal extension problems} 


The discussion up to this point has been purely based on general properties of the heat semigroup $P_t = e^{-t\mathscr L}$ (also the Bessel semigroup plays an important role): no geometry has appeared so far. To see CR geometry one needs to consider the nonlocal operators $\mathscr L_s$. The meeting point of the two operators $\mathscr L^s$ and $\mathscr L_s$ is at $s=1$. It is in fact  worth noting that when $s=1$ in \eqref{bfm},\eqref{bfmH} one obtains $\mathscr L^1 =\mathscr L_1 = - \mathscr L$, whereas by a standard asymptotic analysis one can prove from \eqref{bala} that $\underset{s\nearrow 1}{\lim}\ \mathscr L^s = - \mathscr L$. For $0<s<1$, $\mathscr L^s$ and $\mathscr L_s$ are quite different, and they have rather different histories.

Intertwining operators of order $s$ in the CR-sphere appeared in the literature in the study of representations of automorphisms of the odd-dimensional sphere within the context of representation theory of semisimple Lie groups (we refer the reader to the treatments given in \cite{KS, JW, Co, BOO}). In \cite{Gra} Graham studied  geometrical and analytical properties of integer powers of operators possessing conformal CR-invariances in the CR-sphere, and (via Cayley transform) in the Heisenberg group. He also showed in \cite[Theorem 3.3]{Gra} that the appropriate power of the Folland's gauge function is the fundamental solution of such operators. A general treatment in CR manifolds of conformal integer powers of sub-Laplacians was developed in \cite{GG}. Conformal operators of fractional order were investigated by Branson, Fontana and Morpurgo in \cite{BFM} and definition of $\mathscr L_s$ \eqref{bfm} is due to them. We refer the reader to their work for a deeper insight into the conformal aspects of $\mathscr L_s$ and the role of these nonlocal operators in the study of optimal functional inequalities.

In their work \cite{FGMT} Frank, Gonzalez, Monticelli and Tan have used harmonic analysis and scattering theory to solve the so-called extension problem for the conformal fractional powers \eqref{bfm} of the horizontal Laplacian in $\Hn$: given $f\in C^\infty_0(\Hn)$, find a function $F\in C^\infty(\Hn\times (0,\infty))$ such that
\begin{equation}\label{Ext}
\begin{cases}
\frac{\p^2 F}{\p y^2} + \frac{1-2s}{y} \frac{\p F}{\p y} + \frac{y^2}4 \frac{\p^2 F}{\p \sigma^2} + \mathscr L F = 0,
\\
F((z,\sigma),0) = f(z,\sigma).
\end{cases}
\end{equation}
Comparing \eqref{Ext} with its non-conformal counterpart \eqref{extLsopras} one notes the additional term $\frac{y^2}4 \frac{\p^2 U}{\p \sigma^2}$. This term makes the analysis of \eqref{Ext} completely different from that of the problem \eqref{extHs} since now the extension semigroup is no longer driven by reflected Brownian motion on the  half-line $y>0$, i.e., the Bessel semigroup with infinitesimal generator $\frac{\p^2}{\p y^2} + \frac{1-2s}{y} \frac{\p}{\p y}$ with Neumann condition in $y=0$. Instead, the driving semigroup is now two-dimensional, and the relevant infinitesimal generator is the partial differential operator of Baouendi-Grushin type $\frac{\p^2 }{\p y^2} + \frac{1-2s}{y} \frac{\p }{\p y} + \frac{y^2}4 \frac{\p^2 }{\p \sigma^2}$. Unlike what happens for \eqref{extHs}, see \eqref{qsopras} above, this differential operator cannot be uncoupled from $\mathscr L$ since they both contain differentiation with respect to the variable $\sigma$. While these aspects will be further clarified below, here we mention that the additional basic property of the solution $F$ of \eqref{Ext} is the following weighted Dirichlet-to-Neumann relation proved in \cite[Theorem 1.1]{FGMT}
\[
\mathscr L_s f(z,\sigma) = c(s) \underset{y \to 0^+}{\lim} y^{1-2s} \frac{\p F}{\p y}((z,\sigma),y).
\] 
This remarkable result represents the conformal counterpart of \eqref{DNell} above.

Despite the essential differences between the extension problems for $\mathscr L^s$ and $\mathscr L_s$, our intent in the present work is to further exploit the analogies, by following an approach different from that in \cite{FGMT} and inspired by the works \cite{RTaim, RT}. With this objective in mind, similarly to \eqref{extHs} we now turn to considering, instead of \eqref{Ext}, the extension problem for the nonlocal evolution operator $(\p_t - \mathscr L)_s$.
In the opening of the section we have stressed an important aspect of the analysis of the non-conformal powers $\mathscr L^s$: the fact that both Balakrishnan's formula \eqref{bala} and the Riesz potential \eqref{rpotG0} are based on the same heat semigroup $P_t$. As we have said, this is due to the equalities in \eqref{masomenosopras}, which as we now show are no longer valid in the conformal case. In their cited works \cite{RTaim, RT} Roncal and Thangavelu considered the following parabolic extension problem, which represents the conformal counterpart of \eqref{extHs}:
given a function $u\in C^\infty_0(\bG\times \R_t)$ find a function $U\in C^\infty(\bG\times  \R_t \times \R^+_y)$ such that
\begin{equation}\label{parext}
\begin{cases}
\mathfrak P_{(s)} U \overset{def}{=} \frac{\p^2 U}{\p y^2}  + \frac{1-2s}y \frac{\p U}{\p y} + \frac{y^2}4 \Delta_\sigma U+ \mathscr L U - \frac{\p U}{\p t} = 0,\ \ \ \ \ \text{in}\ \bG\times  \R_t \times \R^+_y,
\\
U(z,\sigma,t,0) = u(z,\sigma,t).
\end{cases}
\end{equation}
Using \eqref{parext} these authors discovered a remarkable Poisson kernel for the problem \eqref{Ext} which represents the conformal counterpart of \eqref{Psenzatempo}. 

To explain their result from the perspective of partial differential equations, and at the same time motivate ours, we now make the crucial observation that, in a group of Heisenberg type $\bG$, the relevant heat equation associated with the problem \eqref{parext} is
\begin{equation}\label{pde}
\Delta_w U + \frac{|w|^2}4 \Delta_\sigma U + \mathscr L U - \p_t U= 0,
\end{equation}
where now $(z,\sigma)\in \bG$, $t>0$. Similarly to \eqref{extHs}, here we are thinking of the variable $w$ as running in the space with fractal dimension $\R^{2(1-s)}$. The link between \eqref{parext} and \eqref{pde} is readily seen by observing that, if $y = |w|$, then on a function $u(w) = \psi(y)$ we have $\Delta_w u = \p_{yy} \psi + \frac {1-2s}{y}\p_y \psi$. Since $\bG$ is of Heisenberg type, by \eqref{L2}  its horizontal Laplacian is given by 
\[
\mathscr L = \Delta_z + \frac{|z|^2}{4} \Delta_\sigma + \sum_{\ell = 1}^k \Theta_\ell \p_{\sigma_\ell},
\]
where $\Theta_\ell = \sum_{s=1}^m \langle J(\ve_\ell)z,e_s\rangle \p_{z_s}$ and $J:V_2 \to \operatorname{End}(V_1)$ is defined by \eqref{kap}.
If one looks for solutions $U$ of \eqref{pde} which are spherically symmetric in the variable $z\in \R^m$, then by \eqref{theta} we have $\Theta_\ell U = 0$ for every $\ell=1,...,k$, and in view of \eqref{salah} the pde \eqref{pde} becomes 
\begin{equation}\label{salahpar}
\Delta_w U + \Delta_z U + \frac{|w|^2 + |z|^2}4 \Delta_\sigma U - \p_t U= 0.
\end{equation}
Remarkably, this is a parabolic Baouendi-Grushin equation in $\R^{m+2(1-s)}\times \R^k\times (0,\infty)$ whose fundamental solution we can explicitly compute \footnote{We recall here that in his 1967 Ph.D. Dissertation \cite{Ba} Baouendi first studied the Dirichlet problem in $L^2$ for a class of degenerate elliptic operators that includes the following prototype
\begin{equation}\label{B}
\mathscr B = \Delta_w + \frac{|w|^2}4 \Delta_\sigma,
\end{equation}
where $(w,\sigma)\in \Rn\times \R^k$. Subsequently, Grushin studied in \cite{Gr1, Gr2} some questions of hypoellipticity connected with the model operator in \eqref{B}.}.  
We will need the following result, for whose proof we refer to \cite{GTstep2}.

\begin{proposition}\label{P:BGpar}
Given a function $f\in C^\infty_0(\R^{n+k})$, the solution of the Cauchy problem  
$$\begin{cases}
\p_t u - \Delta_w u - \frac{|w|^2}4 \Delta_\sigma u = 0\ \ \ \ \ \ \ \emph{in}\ \ \R^{n+k}\times (0,\infty),
\\
u((w,\sigma),0) = f(w,\sigma),
\end{cases}$$
is given by the formula
$$u((w,\sigma),t) = \int_{\Rn} \int_{\R^k} q((w,\sigma),(w',\sigma'),t) f(w',\sigma') dw' d\sigma',$$
where 
\begin{align}\label{parBfs2}
 q((w,\sigma),(w',\sigma'),t) & =  \frac{2^k}{(4\pi t)^{\frac{n}2 +k}} \int_{\R^k} e^{- \frac it \langle \la,\sigma'-\sigma\rangle}   \left(\frac{|\la|}{\sinh |\la|}\right)^{\frac n2} 
\\
& \times  e^{-\frac{|\la|}{4t \tanh |\la|} ((|w|^2 +|w'|^2) - 2 \langle w,w'\rangle \sech |\la|)} d\la.
\notag\end{align}
\end{proposition}
If we take $n = m + 2(1-s)$ in the previous proposition and we keep in mind that $\frac n2 = \frac{m+2(1-s)}2 = \frac m2 + 1-s$, we notice that when the pole $(w',\sigma')$ in \eqref{parBfs2} is fixed at the origin, and we set $y = |w|$, then such fundamental solution is given by  
\begin{align}\label{parBfs}
 q_{(s)}((z,\sigma),t,y) & =  \frac{2^k}{(4\pi t)^{\frac{m}2 +k +1-s}} \int_{\R^k} e^{- \frac it \langle \sigma,\la\rangle}   \left(\frac{|\la|}{\sinh |\la|}\right)^{\frac m2+1-s} e^{-\frac{|z|^2 +y^2}{4t}\frac{|\la|}{\tanh |\la|}} d\la.
\end{align}
The function \eqref{parBfs} plays an important role in the remainder of this paper. It represents the conformal counterpart of the simpler looking fundamental solution $q^{(s)}((z,\sigma),t,y)$ in \eqref{qsopras}. The essential difference between the two is that whereas for $q^{(s)}$ the extension variable $y>0$ and the variable $g \in \bG$ are uncoupled, causing $q^{(s)}$ to be a product of the two fundamental solutions $g^{(s)}(y,t)$ and $p(g,t)$, for $q_{(s)}$ this is not the case. Since, as we have previously mentioned, the vertical variable $\sigma\in \R^k$ appears both in the heat operator $\p_t -\mathscr L$ and in the extension part of \eqref{extHs},
\begin{equation}\label{extpart}
\frac{\p^2}{\p y^2} + \frac{1-2s}{y} \frac{\p }{\p y} + \frac{y^2}4 \Delta_\sigma - \frac{\p }{\p t},
\end{equation}
the fundamental solution \eqref{parBfs} is not a product, but a partial convolution in the central variable. Indeed, it is not difficult to recognise that 
$$q_{(s)}((z,\sigma),y,t)  = p_{(s)}(y,\cdot,t) \star p(z,\cdot,t)(\sigma) = \int_{\R^k} p_{(s)}(y,\eta,t) p(z,\sigma - \eta,t) d\eta,$$
where $p(z,\sigma,t)$ is the Hulanicki-Gaveau-Cygan heat kernel defined by \eqref{pipposemprepiupiccoloHoo}, and we have indicated with $p_{(s)}(y,\sigma,t)$ the fundamental solution (with pole in the origin) of \eqref{extpart}. In conclusion, the function $q_{(s)}$ is defined in the thick space $\bG \times \R^+_{t}\times \R^+_y$, and solves the  parabolic extension differential operator $\mathfrak P_{(s)}$ in \eqref{extHs}.

In analogy with \eqref{sopramenos} its companion function $q_{(-s)}$, which is obtained from \eqref{parBfs} by changing $s$ into $-s$, is the fundamental solution of the intertwined operator
$$\mathfrak P_{(-s)}  \overset{def}{=} \frac{\p^2 }{\p y^2}  + \frac{1+2s}y \frac{\p }{\p y} + \frac{y^2}4 \Delta_\sigma + \mathscr L  - \frac{\p }{\p t}.$$
Precisely as for \eqref{PPsopras}, the link between these two operators is given by 
\begin{equation}\label{intrallazzamento}
\mathfrak P_{(s)}\left(y^{2s} q_{(-s)}\right) = y^{2s} \mathfrak P_{(-s)}\ q_{(-s)} = 0,\ \ \ \ \text{and}\ \ \ \ \mathfrak P_{(-s)}\left(y^{-2s} q_{(s)}\right) = y^{-2s} \mathfrak P_{(s)}\ q_{(s)} = 0.
\end{equation}

Returning to the extension problem \eqref{parext}, in analogy with \eqref{Ksopras} it is natural to expect the function
\begin{equation}\label{Ksottos}
\mathscr P_{(s)}((z,\sigma),t,y) = \frac{4 \pi^{1+s}}{\G(s)} y^{2s} q_{(-s)}((z,\sigma),t,y)
\end{equation}
to be the relevant Poisson kernel. We note that, in view of \eqref{intrallazzamento}, we know that 
\[
\mathfrak P_{(s)} \mathscr P_{(s)} = 0.
\]
With \eqref{Ksottos} in hands, as in \eqref{Psenzatempo} one should accordingly expect 
\begin{equation}\label{Qsottos}
\mathscr Q_{(s)}((z,\sigma),y) = \int_0^\infty \mathscr P_{(s)}((z,\sigma),t,y) dt = \frac{4 \pi^{1+s}}{\G(s)} y^{2s} \int_0^\infty q_{(-s)}((z,\sigma),t,y) dt
\end{equation} 
to be the Poisson kernel for the extension problem for $\mathscr L_s$. As a consequence of our Theorem \ref{T:thick} above, we can explicitly compute \eqref{Qsottos}, obtaining
\begin{equation}\label{Qsottosalt}
\mathscr Q_{(s)}((z,\sigma),y) = \frac{2^{-2s}|\G(-s)|}{\G(s)} C_{(-s)}(m,k) \frac{y^{2s}}{((|z|^2+y^2)^2+16|\sigma|^2)^{\frac{Q+2 s}4}}.
\end{equation}
Using the explicit knowledge of the constant in \eqref{C}, we obtain the following.
\begin{proposition}\label{P:Q}
Let $0<s\le 1$. For every $y>0$ one has
\[
\int_0^\infty \int_{\bG} \mathscr P_{(s)}((z,\sigma),ty) dz d\sigma dt = \int_{\bG} \mathscr Q_{(s)}((z,\sigma),y) dzd\sigma = 1.
\]
\end{proposition}
One should compare Proposition \ref{P:Q} with its non-conformal counterpart in \eqref{pescionone} and \eqref{pescetto}. 
We mention that the right-hand side of formula \eqref{Qsottosalt} is exactly the Poisson kernel found by Roncal and Thangavelu in (1.7) of their \cite[Proposition 4.2]{RTaim} and \cite[Theorems 1.2 \& 3.2]{RT}. With such Poisson kernel they are able to establish, among other things, 
a conformal counterpart of the well-known formula of M. Riesz
\[
(-\Delta)^s u(x) = \gamma(n,s) \operatorname{PV} \int_{\Rn} \frac{u(x) - u(y)}{|x-y|^{n+2s}} dy.
\]

If we now return to the fundamental solution \eqref{parBfs} of the extension operator $\mathfrak P^{(s)}$ in \eqref{parext} and its companion $q_{(-s)}$, with such functions in hands we now define
\begin{equation}\label{masomenosottos}
\mathscr K_{(s)}((z,\sigma),t) = (4\pi t)^{1-s} q_{(s)}((z,\sigma),t,0),\ \ \  \mathscr K_{(-s)}((z,\sigma),t)= (4\pi t)^{1+s} q_{(-s)}((z,\sigma),t,0).
\end{equation}
One should note the difference between \eqref{masomenosottos} and their non-conformal counterparts in \eqref{masomenosopras}.
From \eqref{parBfs} it is immediate to recognise that
\begin{equation}\label{poissono}
\mathscr K_{(s)}((z,\sigma),t) = \frac{2^k}{\left(4\pi t\right)^{\frac m2+k}}\int_{\R^k} e^{-\frac it \langle \sigma,\la\rangle} \left(\frac{|\la|}{\sinh |\la|}\right)^{\frac m2+1-s}  e^{- \frac{|z|^2}{4t} \frac{|\la|}{\tanh |\la|}}d\la,\ \ \ \ 0<s\le 1.
\end{equation}
This finally explains the function \eqref{poisson} in the introduction. 
Note that $\mathscr K_{(s)}$  coincides with the heat kernel $p((z,\sigma),t)$ in \eqref{pipposemprepiupiccoloHoo} in the local case $s = 1$.
 It should be obvious that $\mathscr K_{(-s)}((z,\sigma),t)$ admits a similar integral representation if one changes $s$ into $-s$ in \eqref{poissono}. 
 
 With these definitions in place, we now see that the Roncal-Thangavelu kernel $\mathcal K^s_t(z,\sigma)$ in \cite[Formula (2.18)]{RTaim} is precisely $\mathscr K_{(-s)}((z,\sigma),t)$. By slightly abusing the notation, we now let
\[
\mathscr K_{(\pm s)}((z,\sigma),(\zeta,\tau),t) \overset{def}{=} \mathscr K_{(\pm s)}((\zeta,\tau)^{-1}\circ (z,\sigma),t),
\]
and define two modified heat flows in $\bG\times \R^+_t$ by letting
$$P_{(\pm s),t} u(z,\sigma) = u \star \mathscr K_{(\pm s)}(\cdot,t)(z,\sigma) = \int_{\bG} \mathscr K_{(\pm s)}((z,\sigma),(\zeta,\tau),t) u(\zeta,\tau) d\zeta d\tau.$$
We recall that if $u, v$ are two functions in $\bG$, then their group convolution is defined by
\[
u\star v(g) = \int_{\bG} u(g\circ (g')^{-1}) v(g') dg' = \int_{\bG} v((g')^{-1}\circ g) u(g') dg'.
\]



\section{The main inversion theorem for $\mathscr L_s$}\label{S:inv}

This section is primarily devoted to proving Theorem \ref{T:I2sdue0}. To motivate our approach to such result, we begin with establishing a preliminary fact about the non-conformal operators $\mathscr L^s$ and $\mathscr I^{(2s)}$ introduced respectively in \eqref{bala} and \eqref{rpotG0}. We mention that, although we have not explicitly defined these geometric ambients, the next proposition holds in a Carnot group of arbitrary step.  

\begin{proposition}\label{P:Isuno}
For every $0<s<1$ one has
\[
\mathscr I^{(2s)} \circ \mathscr L^s = \mathscr L^s \circ \mathscr I^{(2s)} = I.
\] 
In particular, this says that the fundamental solution of $\mathscr L^s$ with pole at the group identity is given by the kernel $\mathscr E^{(s)}(g)$  in \eqref{EsH0}.
\end{proposition}   

\begin{proof}
We first observe that, for $u\in C^\infty_0(\bG)$ and $g \in \bG$, we can rewrite \eqref{bala} in the following way 
\begin{align*}
& \mathscr L^s u(g) = - \frac{s}{\G(1-s)} \int_0^\infty \frac{1}{t^{1+s}} (P_t u(g) - u(g)) dt
\\
& = - \frac{s}{\G(1-s)} \int_0^\infty \frac{1}{t^{1+s}} \int_0^t \frac{d}{d\tau} P_\tau u(g) d\tau dt = - \frac{s}{\G(1-s)} \int_0^\infty \frac{1}{t^{1+s}} \int_0^t \mathscr L P_\tau u(g) d\tau dt
\\
& = - \frac{s}{\G(1-s)} \int_0^\infty \mathscr L P_\tau u(g) \int_\tau^\infty  \frac{dt}{t^{1+s}} d\tau = - \frac{1}{\G(1-s)} \int_0^\infty \frac{1}{\tau^{s}} \mathscr L P_\tau u(g)   d\tau. 
\end{align*}
With this formula in hands, we obtain
\begin{align*}
& \mathscr I^{(2s)} (\mathscr L^s u)(g) = - \frac{1}{\G(s)\G(1-s)} \int_0^\infty t^{s} \int_0^\infty \frac{1}{\tau^{s}} \mathscr L P_{t+\tau} u(g)  d\tau \frac{dt}t
\\
& (\text{change of variable}\ t = \tau \rho,\ \frac{dt}t = \frac{d\rho}\rho)
\\
&  = - \frac{1}{\G(s)\G(1-s)} \int_0^\infty \rho^{s-1} \int_0^\infty  \mathscr L P_{(1+\rho)\tau} u(g)  d\tau d\rho
\\
& (\text{change of variable}\ \sigma = (1+\rho)\tau,\ d\sigma = (1+\rho) d\tau)
\\
& = - \frac{1}{\G(s)\G(1-s)} \int_0^\infty \frac{\rho^{s-1}}{1+\rho} d\rho \int_0^\infty  \mathscr L P_{\sigma} u(g)  d\sigma
\\
& = - \int_0^\infty  \frac{d}{d\sigma} P_{\sigma} u(g)  d\sigma = u(g),
\end{align*}
where in the second to the last equality we have used the following well-known representation of the beta function
\[
B(x,y) = \int_0^\infty \frac{\rho^{x-1}}{(1+\rho)^{x+y}} d\rho,
\]
which gives for $0<s<1$
\[
 \int_0^\infty \frac{\rho^{s-1}}{1+\rho} d\rho = B(s,1-s) = \G(s)\G(1-s).
 \]
In the last equality, instead, we have used the hypercontractive estimate (here, $Q = m+2k$)
\[
|P_t u(g)| \le \frac{C}{t^{Q/2}} ||u||_{L^\infty(\bG)},
\] 
which implies $\underset{t\to \infty}{\lim} P_t u(g) = 0$.

\end{proof}

Our next objective is to establish the conformal analogue of the inversion formula of Proposition \ref{P:Isuno}. There are two aspects that play a big role in the proof of the latter: (a) the already noted fact that the same operator, $P_t$, occurs in both the expression for $\mathscr I^{(2s)}$ and $\mathscr L^s$; (b) the semigroup property of $P_t$. Things are not as simple in the conformal case. Since in the definition \eqref{conformalriesz} of $\mathscr I_{(2s)}$ the operator $P_{(s), t}$ appears, whereas in that \eqref{RTs} of $\mathscr L_s$ we have the different operator $P_{(-s), t}$, we need some form of replacement of the semigroup property for the composition of such intertwining operators and their defining kernels $\mathscr K_{(\pm s)}$ in \eqref{poissono}. This is precisely the purpose of the following Lemma \ref{kernallazzamento} and Corollary \ref{heatrallazzamento}.

{\allowdisplaybreaks
\begin{lemma}\label{kernallazzamento}
Fix $s\in (0,1)$, $g\in\bG$ and $t,\tau>0$. Then, we have
\begin{align}\label{eccola}
\int_{\bG}\mathscr K_{(-s)}((g')^{-1}\circ g,\tau)\mathscr K_{(s)}(g',t)\, dg'=&\int_{\R^k} e^{2\pi i \langle \sigma,\la\rangle}\left(\frac{2\pi t|\la|}{\sinh 2\pi t |\la|}\right)^{1-s}\left(\frac{2\pi \tau|\la|}{\sinh 2\pi\tau |\la|}\right)^{1+s} \times\\
&\times \left(\frac{|\la|}{2\sinh 2\pi (t+\tau) |\la|}\right)^{\frac m2} e^{-\frac \pi2 |z|^2 \frac{|\la|}{\tanh 2\pi (t+\tau) |\la|}}d\la .\nonumber
\end{align}
\end{lemma}
}

{\allowdisplaybreaks
\begin{proof}
If we keep  \eqref{masomenosottos} in mind, it is clear that \eqref{eccola} is equivalent to showing that
\begin{align}\label{eccolaQ}
\int_{\bG} q_{(-s)}((g')^{-1}&\circ g,\tau,0) q_{(s)}(g',t,0)\, dg'=\int_{\R^k} e^{2\pi i \langle \sigma,\la\rangle}\left(\frac{|\la|}{2\sinh 2\pi t |\la|}\right)^{1-s}\times\\
&\times\left(\frac{|\la|}{2\sinh 2\pi\tau |\la|}\right)^{1+s}\left(\frac{|\la|}{2\sinh 2\pi (t+\tau) |\la|}\right)^{\frac m2} e^{-\frac \pi2 |z|^2 \frac{|\la|}{\tanh 2\pi (t+\tau) |\la|}}d\la .\nonumber
\end{align}
To establish \eqref{eccolaQ}, we first fix $t, \tau>0$ and $z\in\R^m$. By partial Fourier transform with respect to the vertical variable $\sigma\in \R^k$, we observe that the desired conclusion \eqref{eccolaQ} will be true if we can prove that for all $\lambda\in\R^k$ the following holds
\begin{align}\label{claimFourier}
&\int_{\R^k} e^{-2\pi i \langle \sigma,\la\rangle}\int_{\bG} q_{(-s)}((g')^{-1}\circ g,\tau,0) q_{(s)}(g',t,0)\, dg'd\sigma\\
&=\left(\frac{|\la|}{2\sinh 2\pi t |\la|}\right)^{1-s}\left(\frac{|\la|}{2\sinh 2\pi\tau |\la|}\right)^{1+s}\left(\frac{|\la|}{2\sinh 2\pi (t+\tau) |\la|}\right)^{\frac m2} e^{-\frac \pi2 |z|^2 \frac{|\la|}{\tanh 2\pi (t+\tau) |\la|}}.
\nonumber
\end{align}
The rest of the proof of the lemma will thus be devoted to establishing \eqref{claimFourier}. Recalling the definition of $q_{(\pm s)}$ in \eqref{parBfs}, by a simple change of variable we see that
\begin{equation}\label{defaltq}
q_{(\pm s)}(g',t,0)  = \int_{\R^k} e^{2\pi i \langle \sigma',\la\rangle} \left(\frac{|\la|}{2\sinh 2\pi t |\la|}\right)^{\frac m2+1\mp s}  e^{-\frac \pi2 |z'|^2 \frac{|\la|}{\tanh 2\pi t |\la|}}d\la.
\end{equation}
Using \eqref{grouplaw}, we also have
\begin{align}\label{convqs}
&q_{(\pm s)}((g')^{-1}\circ g,\tau,0) \\
&= \int_{\R^k} e^{2\pi i \left(\langle \sigma-\sigma',\la\rangle + \frac 12\left\langle J(\lambda)z,z' \right\rangle\right)} \left(\frac{|\la|}{2\sinh 2\pi \tau |\la|}\right)^{\frac m2+1\mp s}  e^{-\frac \pi2 |z-z'|^2 \frac{|\la|}{\tanh 2\pi \tau |\la|}}d\la.\nonumber
\end{align}
Next, we note that by \eqref{defaltq}-\eqref{convqs}, and applying twice the Fourier inversion formula, for every fixed $\la\in \R^k$ we can rewrite the left-hand side of \eqref{claimFourier} in the following way
\begin{align}\label{lhs}
&\int_{\R^k} e^{-2\pi i \langle \sigma,\la\rangle}\int_{\bG} q_{(-s)}((g')^{-1}\circ g,\tau,0) q_{(s)}(g',t,0)\, dg'd\sigma\\
&=\int_{\R^k}\int_{\R^k}\int_{\R^k}\int_{\bG} e^{-2\pi i \langle \sigma,\la\rangle} e^{2\pi i \langle \sigma -\sigma',\mu\rangle}e^{\pi i \langle J(\mu)z,z'\rangle} e^{2\pi i \langle \sigma',\omega\rangle}\left(\frac{|\mu|}{2\sinh 2\pi \tau |\mu|}\right)^{\frac m2+1+s}\times\nonumber\\  
&\times \left(\frac{|\omega|}{2\sinh 2\pi t |\omega|}\right)^{\frac m2+1-s} \times e^{-\frac \pi2 |z-z'|^2 \frac{|\mu|}{\tanh 2\pi \tau |\mu|}} e^{-\frac \pi2 |z'|^2 \frac{|\omega|}{\tanh 2\pi t |\omega|}} d g' d\omega d\mu d\sigma\nonumber\\
&=\left(\frac{|\lambda|}{2\sinh 2\pi \tau |\lambda|}\right)^{\frac m2+1+s}\int_{\R^k}\int_{\bG} e^{-2\pi i \langle \sigma',\lambda\rangle}e^{\pi i \langle J(\lambda)z,z'\rangle} e^{2\pi i \langle \sigma',\omega\rangle}  \left(\frac{|\omega|}{2\sinh 2\pi t |\omega|}\right)^{\frac m2+1-s} \times\nonumber\\
&\times e^{-\frac \pi2 |z-z'|^2 \frac{|\lambda|}{\tanh 2\pi \tau |\lambda|}} e^{-\frac \pi2 |z'|^2 \frac{|\omega|}{\tanh 2\pi t |\omega|}} d g' d\omega \nonumber\\
&=\left(\frac{|\lambda|}{2\sinh 2\pi \tau |\lambda|}\right)^{\frac m2+1+s}\left(\frac{|\lambda|}{2\sinh 2\pi t |\lambda|}\right)^{\frac m2+1-s}\times\nonumber\\
&\times\int_{\R^m} e^{\pi i \langle J(\lambda)z,z'\rangle} e^{-\frac \pi2 |z-z'|^2 \frac{|\lambda|}{\tanh 2\pi \tau |\lambda|}} e^{-\frac \pi2 |z'|^2 \frac{|\lambda|}{\tanh 2\pi t |\lambda|}} d z'.\nonumber
\end{align}

We now notice that a simple computation yields
\begin{align*}
&\frac{|z-z'|^2}{\tanh 2\pi \tau |\lambda|}+\frac{|z'|^2}{\tanh 2\pi t |\lambda|}
\\
&=\left(\frac{1}{\tanh 2\pi \tau |\lambda|}+\frac{1}{\tanh 2\pi t |\lambda|}\right)\left|z' -\frac{(\tanh 2\pi \tau |\lambda|)^{-1}}{(\tanh 2\pi \tau |\lambda|)^{-1}+(\tanh 2\pi t |\lambda|)^{-1}}  z\right|^2 
\\
&+ \frac{1}{\tanh 2\pi \tau |\lambda| + \tanh 2\pi t |\lambda|}|z|^2.
\end{align*}
Using this identity in \eqref{lhs} we  deduce
\begin{align}\label{provingclaim}
&\int_{\R^k} e^{-2\pi i \langle \sigma,\la\rangle}\int_{\bG} q_{(-s)}((g')^{-1}\circ g,\tau,0) q_{(s)}(g',t,0)\, dg'd\sigma\\
&=\left(\frac{|\lambda|}{2\sinh 2\pi \tau |\lambda|}\right)^{\frac m2+1+s}\left(\frac{|\lambda|}{2\sinh 2\pi t |\lambda|}\right)^{\frac m2+1-s}e^{-\frac \pi2 |z|^2 \frac{|\lambda|}{\tanh 2\pi \tau |\lambda| + \tanh 2\pi t |\lambda|}}
\times\nonumber\\
&\times\int_{\R^m} e^{\pi i \langle J(\lambda)z,z'\rangle} e^{-\frac \pi2 \left|z' -\frac{(\tanh 2\pi \tau |\lambda|)^{-1}}{(\tanh 2\pi \tau |\lambda|)^{-1}+(\tanh 2\pi t |\lambda|)^{-1}}  z\right|^2 \left(\frac{|\la|}{\tanh 2\pi \tau |\lambda|}+\frac{|\la|}{\tanh 2\pi t |\lambda|}\right)} d z'.\nonumber\\
&=\left(\frac{|\lambda|}{2\sinh 2\pi \tau |\lambda|}\right)^{\frac m2+1+s}\left(\frac{|\lambda|}{2\sinh 2\pi t |\lambda|}\right)^{\frac m2+1-s}e^{-\frac \pi2 |z|^2 \frac{|\lambda|}{\tanh 2\pi \tau |\lambda| + \tanh 2\pi t |\lambda|}}
\times\nonumber\\
&\times\int_{\R^m} e^{\pi i \langle J(\lambda)z,\xi\rangle} e^{-\frac \pi2 |\xi|^2 \left(\frac{|\la|}{\tanh 2\pi \tau |\lambda|}+\frac{|\la|}{\tanh 2\pi t |\lambda|}\right)} d \xi\nonumber\\
&=\left(\frac{|\lambda|}{2\sinh 2\pi \tau |\lambda|}\right)^{\frac m2+1+s}\left(\frac{|\lambda|}{2\sinh 2\pi t |\lambda|}\right)^{\frac m2+1-s}\left(\frac{|\la|}{2\tanh 2\pi \tau |\lambda|}+\frac{|\la|}{2\tanh 2\pi t |\lambda|}\right)^{-\frac m2}\times\nonumber\\
&\times e^{-\frac \pi2 |z|^2 \frac{|\lambda|}{\tanh 2\pi \tau |\lambda| + \tanh 2\pi t |\lambda|}}e^{- \frac{\pi|\la|^2  |z|^2}{4} \left(\frac{|\la|}{2\tanh 2\pi \tau |\lambda|}+\frac{|\la|}{2\tanh 2\pi t |\lambda|}\right)^{-1}},
 \nonumber
\end{align}
where we have crucially used the skew-symmetry of $J(\la)$, the fact that $|J(\la)z|=|\la||z|$, and the well-known formula $\int_{\R^m} e^{2\pi i \langle x,\xi\rangle} e^{-\pi \alpha |\xi|^2} d \xi = \alpha^{-\frac m2}e^{-\frac{\pi|x|^2}{\alpha}}$. Finally, if we insert in \eqref{provingclaim} the following two identities:
\begin{align*}
& \left(\frac{|\la|}{2\tanh 2\pi \tau |\lambda|}+\frac{|\la|}{2\tanh 2\pi t |\lambda|}\right)^{-1}=\frac{2\sinh 2\pi \tau |\lambda|}{|\la|}\frac{2\sinh 2\pi t |\lambda|}{|\la|}\frac{|\la|}{2\sinh 2\pi (t+\tau) |\lambda|},
\end{align*}
and
\begin{align*}
&\frac{|\lambda|}{\tanh 2\pi \tau |\lambda| + \tanh 2\pi t |\lambda|}+ \frac{|\la|^2 }{2} \left(\frac{|\la|}{2\tanh 2\pi \tau |\lambda|}+\frac{|\la|}{2\tanh 2\pi t |\lambda|}\right)^{-1}=\\
&=|\la|\frac{\cosh 2\pi \tau |\lambda|\cosh 2\pi t |\lambda| + \sinh 2\pi \tau |\lambda| \sinh 2\pi t |\lambda|}{\sinh 2\pi (t+\tau) |\lambda|}=\frac{|\la|}{\tanh 2\pi (t+\tau) |\lambda|},
\end{align*}
we obtain the desired conclusion \eqref{claimFourier}. This completes the proof of the lemma.
\end{proof}
}

In what follows, it will be helpful to introduce a special notation for the kernel appearing in Lemma \ref{kernallazzamento}. Given any $t,\tau>0$ and $g\in\bG$, we let
\begin{align}\label{nucleooo}
\mathscr K_{(-s,s)}(g,\tau,t)=&\int_{\R^k} e^{2\pi i \langle \sigma,\la\rangle}\left(\frac{2\pi t|\la|}{\sinh 2\pi t |\la|}\right)^{1-s}\left(\frac{2\pi \tau|\la|}{\sinh 2\pi\tau |\la|}\right)^{1+s} \times\\
&\times \left(\frac{|\la|}{2\sinh 2\pi (t+\tau) |\la|}\right)^{\frac m2} e^{-\frac \pi2 |z|^2 \frac{|\la|}{\tanh 2\pi (t+\tau) |\la|}}d\la.\nonumber
\end{align}
Moreover, for $u\in C_0^\infty(\bG)$ we define
\begin{equation}\label{defPmasomeno}
P_{(-s,s),(\tau,t)}u(g)=\int_{\bG}\mathscr K_{(-s,s)}((g')^{-1}\circ g,\tau,t)u(g') dg'.
\end{equation}
With these notations in place we now obtain from Lemma \ref{kernallazzamento} the following.
\begin{corollary}\label{heatrallazzamento}
Let $0<s<1$. For any $t,\tau>0$, we have
$$P_{(-s),\tau}P_{(s),t}=P_{(s),t}P_{(-s),\tau}=P_{(-s,s),(\tau,t)}.$$
\end{corollary}
\begin{proof}
It easily follows from an application of Lemma \ref{kernallazzamento}, the invariance of the Lebesgue measure with respect to the left-translation, and from the following two symmetry properties which are valid for all $g\in\bG$
$$\mathscr K_{(-s,s)}(g,\tau,t)=\mathscr K_{(s,-s)}(g,t,\tau),$$
$$\mathscr K_{(-s,s)}(g,\tau,t)=\mathscr K_{(-s,s)}(g^{-1},\tau,t).$$

\end{proof}

In the course of the proof of Theorem \ref{T:I2sdue0} we will be mimicking that of Proposition \ref{P:Isuno}, and we will perform the following change of variables
\begin{equation}\label{change}
(t,\tau)\mapsto (v,\rho)=\left(t+\tau, \frac{t}{\tau}\right).
\end{equation}
With this in mind, it will be expedient to fix a notation for the operator in \eqref{defPmasomeno} in the new variables $(v,\rho)$. We thus define
\begin{equation}\label{defopinter}
P_{(-s,s),v}^{\rho}\overset{def}{=}P_{(-s,s),(\tau,t)}.
\end{equation}
Next, we establish a couple of useful lemmas.
\begin{lemma}\label{approxid}
Let $0<s<1$. For any function $u\in C^\infty_0(\bG)$ and any $g\in\bG$, one has
\begin{equation}\label{appA}
\underset{t\to 0^+}{\lim} P_{(\pm s),t}u (g) = u(g),
\end{equation}
\begin{equation}\label{appttau}
\underset{v\to 0^+}{\lim} P^{\rho}_{(-s,s),v}u (g) =u(g)\qquad\mbox{ for all }\rho>0.
\end{equation}
\end{lemma}
\begin{proof}
Since the proofs of \eqref{appA} and \eqref{appttau} follow the same lines, we only provide the details of the latter.
By left-translation we can assume, without restriction, that $g=(0,\sigma_0)$ for $\sigma_0\in \R^k$. Applying Fourier transform in the vertical variables, our objective is to prove the following limiting behaviour for every fixed $\lambda\in \R^k$
\begin{equation}\label{todoperapprox}
\hat{u}(0,\lambda)=\underset{v\to 0^+}{\lim} \int_{\R^k}\int_{\bG}e^{2\pi i\langle \sigma_0,\lambda\rangle}\mathscr K_{(-s,s)}\left((z,\sigma-\sigma_0),\frac{v}{1+\rho},\frac{\rho v}{1+\rho}\right)u(z,\sigma) dz d\sigma d\sigma_0.
\end{equation}
Recalling the relevant definitions we have
\begin{align*}
&\int_{\R^k}\int_{\bG}e^{2\pi i\langle \sigma_0,\lambda\rangle}\mathscr K_{(-s,s)}\left((z,\sigma-\sigma_0),\frac{v}{1+\rho},\frac{\rho v}{1+\rho}\right)u(z,\sigma) dz d\sigma d\sigma_0\\
&=\left(\frac{2\pi \frac{v}{1+\rho}|\la|}{\sinh 2\pi \frac{ v}{1+\rho}|\la|}\right)^{1+s}\left(\frac{2\pi \frac{\rho v}{1+\rho} |\la|}{\sinh 2\pi \frac{\rho v}{1+\rho} |\la|}\right)^{1-s} \left(\frac{|\la|}{2\sinh 2\pi v |\la|}\right)^{\frac m2}\int_{\R^m}\hat{u}(z,\lambda)e^{-\frac \pi2 |z|^2 \frac{|\la|}{\tanh 2\pi v |\la|}} dz\nonumber\\
&=\left(\frac{2\pi \frac{v}{1+\rho}|\la|}{\sinh 2\pi \frac{ v}{1+\rho}|\la|}\right)^{1+s}\left(\frac{2\pi \frac{\rho v}{1+\rho} |\la|}{\sinh 2\pi \frac{\rho v}{1+\rho} |\la|}\right)^{1-s} \left(\frac{|\la|}{2\sinh 2\pi v |\la|}\right)^{\frac m2}\left(\frac{2\tanh 2\pi v |\la|}{|\la|}\right)^{\frac m2}\times\nonumber\\
&\times\int_{\R^m}\hat{u}(z,\lambda)\left(4\pi v \frac{\tanh 2\pi v |\la|}{2\pi v |\la|}\right)^{-\frac m2}e^{-\frac \pi2 |z|^2 \frac{|\la|}{\tanh 2\pi v |\la|}} dz\nonumber\\
&=\left(\frac{2\pi \frac{v}{1+\rho}|\la|}{\sinh 2\pi \frac{ v}{1+\rho}|\la|}\right)^{1+s}\left(\frac{2\pi \frac{\rho v}{1+\rho} |\la|}{\sinh 2\pi \frac{\rho v}{1+\rho} |\la|}\right)^{1-s} \left(\frac{1}{\cosh 2\pi v |\la|}\right)^{\frac m2}\times\nonumber\\
&\times\int_{\R^m}\hat{u}(z,\lambda)\left(4\pi v \frac{\tanh 2\pi v |\la|}{2\pi v |\la|}\right)^{-\frac m2}e^{-\frac \pi2 |z|^2 \frac{|\la|}{\tanh 2\pi v |\la|}} dz.
\end{align*}
The last term in the previous chain of identities does converge to $\hat{u}(0,\lambda)$ as $v\to 0^+$: one can in fact recognize in the last integral the convolution with the Gauss-Weierstrass kernel in $\R^m$ (with pole at $z=0$, at time $v \frac{\tanh 2\pi v |\la|}{2\pi v |\la|}$) which is well-known that approximates the identity. This proves \eqref{todoperapprox}. 

\end{proof}

\begin{lemma}
Let $0<s<1$. For any $u\in C_0^\infty(\bG)$ and $g\in\bG$ we have
\begin{equation}\label{appinfA}
\underset{t\to +\infty}{\lim} P_{(\pm s),t}u (g) = 0,
\end{equation}
and
\begin{equation}\label{appinfv}
\underset{v\to +\infty}{\lim} P^\rho_{(-s,s),v}u (g)=0\qquad\mbox{ for any }\,\rho>0.
\end{equation}
\end{lemma}
\begin{proof}
The proofs of \eqref{appinfA} and \eqref{appinfv} follow the same lines. We provide here the details of the proof of \eqref{appinfv}.
Without loss of generality we can assume that $g=0$. If we denote by $\hat{u}$ the partial Fourier transform in the vertical variables, then we can rewrite
\begin{align*}
&P^\rho_{(-s,s),v}u (0)=\int_{\R^m}\int_{\R^k} \hat{u}(z,\lambda)\left(\frac{2\pi \frac{v}{1+\rho}|\la|}{\sinh 2\pi \frac{ v}{1+\rho}|\la|}\right)^{1+s}\left(\frac{2\pi \frac{\rho v}{1+\rho} |\la|}{\sinh 2\pi \frac{\rho v}{1+\rho} |\la|}\right)^{1-s}\times\\
&\times \left(\frac{|\la|}{2\sinh 2\pi v |\la|}\right)^{\frac m2}  e^{-\frac \pi2 |z|^2 \frac{|\la|}{\tanh 2\pi v |\la|}}d\la dz.\nonumber
\end{align*}
Since $\hat{u}$ is in Schwartz class we can pass the limit as $v$ tends to $\infty$ inside the integral and conclude \eqref{appinfv}.

\end{proof}

At a technical level, the striking difference between the proofs of Proposition \ref{P:Isuno} and Theorem \ref{T:I2sdue0} is the dependence on $\rho$ of the operator $P^\rho_{(-s,s),v}$ defined in \eqref{defopinter}. As a matter of fact, we recall that $P^\rho_{(-s,s),v}$ is nothing but the composition of $P_{(-s),\tau}$ and $P_{(s),t}$ in the new variables \eqref{change}, and the nonconformal analogue in Proposition \ref{P:Isuno} is just given by $P_{\tau}P_{t}=P_{\tau + t}=P_v$, which is obviously $\rho$-independent. However, even if $P^\rho_{(-s,s),v}$ depends on $\rho$ in a highly non-trivial way, there is a remarkable hidden cancellation property which is satisfied by a precise weighted average of $\frac{\partial}{\partial\rho}P^\rho_{(-s,s),v}$. Such cancellation is encoded in the vanishing of the integral $A(s,\mu)$ in the next lemma, and its relation with $P^\rho_{(-s,s),v}$ will be clarified in the proof of Theorem \ref{T:I2sdue0}.

\begin{lemma}\label{lemmaMonster}
For any $s\in(-1,1)$ and $\mu>0$ we have
\begin{align}\label{monstersolved}
A(s,\mu)\overset{def}{=}\int_0^\infty\rho^{s-1}&\left[\frac{(1+s)\rho}{1+\rho}\left(\frac{\frac{\mu}{1+\rho}}{\tanh \frac{\mu}{1+\rho}}-1 \right)-\frac{1-s}{1+\rho}\left(\frac{\frac{\rho \mu}{1+\rho}}{\tanh \frac{\rho \mu}{1+\rho}}-1\right)\right]\times\\
&\times\left(\frac{\frac{\rho \mu}{1+\rho}}{\sinh \frac{\rho \mu}{1+\rho}}\right)^{1-s}\left(\frac{\frac{\mu}{1+\rho}}{\sinh \frac{ \mu}{1+\rho}}\right)^{1+s} d\rho\ =\ 0.\nonumber
\end{align}
\end{lemma}
\begin{proof}
Fix $s\in(-1,1)$ and $\mu>0$. Consider the function $h_{s,\mu}:(0,\infty)\mapsto \R$ defined by
$$h_{s,\mu}(\rho)=\frac{\mu}{\sinh\mu}\left(\frac{\sinh \frac{\rho \mu}{1+\rho}}{\sinh \frac{ \mu}{1+\rho}}\right)^s\left[\frac{\mu}{\sinh\mu}\frac{\rho}{(1+\rho)^2}\left(\frac{\sinh \frac{\rho \mu}{1+\rho}}{\sinh \frac{ \mu}{1+\rho}} +2\cosh\mu + \frac{\sinh \frac{ \mu}{1+\rho}}{\sinh \frac{\rho \mu}{1+\rho}} \right)-1\right].$$
By Taylor expansion it is not difficult to verify that
$$
h_{s,\mu}(\rho)=O(\rho^{s+1}) \mbox{ as }\rho\to 0^+\quad\mbox{and}\quad h_{s,\mu}(\rho)=O(\rho^{s-1})\mbox{ as }\rho\to \infty.
$$
Since $s\in(-1,1)$, this implies that
\begin{equation}\label{estremi}
h_{s,\mu}(0^+)=0\qquad\mbox{ and }\qquad h_{s,\mu}(+\infty)=0.
\end{equation}
A straightforward computation shows that
\begin{align*}
&h_{s,\mu}'(\rho)=\frac{\mu}{\sinh\mu}\left(\frac{\sinh \frac{\rho \mu}{1+\rho}}{\sinh \frac{ \mu}{1+\rho}}\right)^{s-1}\times\\
&\times\left[\frac{\mu}{\sinh\mu}\left(\frac{1}{(1+\rho)^2}-\frac{2\rho}{(1+\rho)^3}\right)\left(\frac{\sinh^2 \frac{\rho \mu}{1+\rho}}{\sinh^2 \frac{ \mu}{1+\rho}}+2\cosh\mu\frac{\sinh \frac{\rho \mu}{1+\rho}}{\sinh \frac{ \mu}{1+\rho}} + 1\right)\right.\\
&\left.-\frac{s\mu\sinh\mu}{(1+\rho)^2\sinh^2\frac{ \mu}{1+\rho}}+\frac{\mu^2\rho}{(1+\rho)^4\sinh^2 \frac{ \mu}{1+\rho}}\left((1+s)\frac{\sinh \frac{\rho \mu}{1+\rho}}{\sinh \frac{ \mu}{1+\rho}} +2s\cosh\mu + (s-1)\frac{\sinh \frac{ \mu}{1+\rho}}{\sinh \frac{\rho \mu}{1+\rho}}\right)\right]\\
&=\rho^{s-1}\left(\frac{\frac{\rho \mu}{1+\rho}}{\sinh \frac{\rho \mu}{1+\rho}}\right)^{1-s}\left(\frac{\frac{\mu}{1+\rho}}{\sinh \frac{ \mu}{1+\rho}}\right)^{1+s}\times\\
&\times\left[\frac{1}{\sinh^2\mu}\left(1-\frac{2\rho}{1+\rho}\right)\left(\sinh^2\frac{\rho \mu}{1+\rho}+2\cosh\mu\sinh \frac{\rho \mu}{1+\rho} \sinh \frac{ \mu}{1+\rho} + \sinh^2 \frac{ \mu}{1+\rho}\right)\right.\\
&\left.-s+\frac{\rho\mu}{(1+\rho)^2\sinh\mu}\left((1+s)\frac{\sinh \frac{\rho \mu}{1+\rho}}{\sinh \frac{ \mu}{1+\rho}} +2s\cosh\mu + (s-1)\frac{\sinh \frac{ \mu}{1+\rho}}{\sinh \frac{\rho \mu}{1+\rho}}\right)\right].
\end{align*}
If we substitute in the previous computation the following three identities:
$$
\sinh^2\mu=\sinh^2\frac{\rho \mu}{1+\rho}+2\cosh\mu\sinh \frac{\rho \mu}{1+\rho}\sinh \frac{ \mu}{1+\rho} + \sinh^2 \frac{ \mu}{1+\rho},
$$
$$
\frac{\sinh \frac{\rho \mu}{1+\rho}}{\sinh \frac{ \mu}{1+\rho}}+\cosh\mu=\frac{\sinh\mu}{\tanh \frac{ \mu}{1+\rho}},
$$
and
$$
\cosh\mu+\frac{\sinh \frac{ \mu}{1+\rho}}{\sinh \frac{\rho \mu}{1+\rho}}=\frac{\sinh\mu}{\tanh \frac{\rho \mu}{1+\rho}},
$$
we obtain
\begin{align}\label{derivacca}
h_{s,\mu}'(\rho)&=\rho^{s-1}\left(\frac{\frac{\rho \mu}{1+\rho}}{\sinh \frac{\rho \mu}{1+\rho}}\right)^{1-s}\left(\frac{\frac{\mu}{1+\rho}}{\sinh \frac{ \mu}{1+\rho}}\right)^{1+s}\left[1-s-\frac{2\rho}{1+\rho}\right.
\\
&\left.+\frac{\mu\rho}{(1+\rho)^2}\left(\frac{1+s}{\tanh \frac{\mu}{1+\rho}}-\frac{1-s}{\tanh \frac{\rho \mu}{1+\rho}}
\right)\right].
\nonumber
\end{align}
Comparing the terms in \eqref{derivacca} with the ones in the definition of $A(s,\mu)$ and using \eqref{estremi}, we finally obtain
$$
A(s,\mu)=\int_{0}^{\infty}h_{s,\mu}'(\rho)d\rho =0,
$$
which is the desired conclusion.

\end{proof}

We are finally in a position to prove Theorem \ref{T:I2sdue0}.

\begin{proof}[Proof of Theorem \ref{T:I2sdue0}]
Fix $u\in C^\infty_0(\bG)$ and $s\in (0,1)$. For any $g \in \bG$, using the definition \eqref{RTs} and the limiting behavior in \eqref{appA} we can write
\begin{align}\label{baladeriva}
& \mathscr L_s u(g) = \frac{-s}{\G(1-s)} \int_0^\infty t^{-s-1}\int_0^t \frac{\partial}{\partial\tau}\left(P_{(-s),\tau}u (g)\right) d\tau dt
\\
&= \frac{-s }{\G(1-s)} \int_0^\infty \int_{\tau}^{\infty} t^{-s-1} \frac{\partial}{\partial\tau}\left(P_{(-s),\tau}u (g)\right) dt d\tau\nonumber
\\
&= \frac{-1}{\G(1-s)} \int_0^\infty \tau^{-s} \frac{\partial}{\partial\tau}\left(P_{(-s),\tau}u (g) \right) d\tau. \nonumber
\end{align}
From \eqref{baladeriva} and the definition \eqref{conformalriesz} we then obtain
\begin{align}\label{compos}
& \mathscr I_{(2s)} (\mathscr L_s u)(g)= \frac{1}{\G(s)} \int_0^\infty t^{s}P_{(s),t}\left(\mathscr L_s u\right)(g)  \frac{dt}t\\
&= -\frac{1}{\G(s)\G(1-s)} \int_0^\infty\int_0^\infty \frac{t^s}{\tau^{s}} P_{(s),t}\left(\frac{\partial}{\partial\tau}\left(P_{(-s),\tau}u\right)\right)(g)d\tau \frac{dt}t\nonumber\\
&= -\frac{1}{\G(s)\G(1-s)} \int_0^\infty\int_0^\infty \frac{t^s}{\tau^{s}} \frac{\partial}{\partial\tau}\left(P_{(s),t}P_{(-s),\tau}u (g)\right) d\tau \frac{dt}t.\nonumber
\end{align}
We notice that one can justify the chain of equalities in \eqref{baladeriva}-\eqref{compos} by recognising that the involved kernels (which are explicit) enjoy the right summability properties.
 
By \eqref{compos} and Corollary \ref{heatrallazzamento}, we thus deduce
$$
\mathscr I_{(2s)} (\mathscr L_s u)(g)=-\frac{1}{\G(s)\G(1-s)} \int_0^\infty\int_0^\infty \frac{t^s}{\tau^{s}} \frac{\partial}{\partial\tau}\left(P_{(-s,s),(\tau,t)} u (g)\right) d\tau \frac{dt}t.
$$
We now perform the change of variables in \eqref{change}, and we use that $dtd\tau=\frac{v}{(1+\rho)^2}dv d\rho$ and
$$
\frac{\partial}{\partial\tau}=\frac{\partial}{\partial v} - \frac{\rho(1+\rho)}{v}\frac{\partial}{\partial\rho}.
$$
This leads to the following identity
\begin{equation}\label{theone}
\mathscr I_{(2s)} (\mathscr L_s u)(g)=-\frac{1}{\G(s)\G(1-s)} \int_0^\infty\int_0^\infty \frac{\rho^{s-1}}{1+\rho}\left(\frac{\partial}{\partial v} - \frac{\rho(1+\rho)}{v}\frac{\partial}{\partial\rho}\right)P^\rho_{(-s,s),v}u (g) dvd\rho.
\end{equation}
Exploiting \eqref{appttau} and \eqref{appinfv} we realize that
$$
-\frac{1}{\G(s)\G(1-s)} \int_0^\infty\int_0^\infty \frac{\rho^{s-1}}{1+\rho}\frac{\partial}{\partial v}P^\rho_{(-s,s),v}u (g) dvd\rho=\frac{u(g)}{\G(s)\G(1-s)} \int_0^\infty \frac{\rho^{s-1}}{1+\rho}d\rho=u(g).
$$
Therefore, from \eqref{theone} we conclude the following
\begin{equation}\label{hereweare}
\mathscr I_{(2s)} (\mathscr L_s u)(g)=u(g)+\frac{s}{\G(1+s)\G(1-s)} \int_0^\infty\int_0^\infty \frac{\rho^{s}}{v}\frac{\partial}{\partial\rho}P^\rho_{(-s,s),v}u (g) dvd\rho.
\end{equation}
Having in mind \eqref{nucleooo}, a direct computation yields
\begin{align*}
&\frac{\partial}{\partial\rho}P^\rho_{(-s,s),v}u (g)\\
&=\int_{\bG}u(g\circ g')\int_{\R^k} e^{2\pi i \langle \sigma',\la\rangle}\left[\frac{1+s}{1+\rho}\left( \frac{2\pi \frac{v}{1+\rho} |\la|}{\tanh 2\pi \frac{v}{1+\rho} |\la|}-1 \right)-\frac{1-s}{\rho(1+\rho)}\left( \frac{2\pi \frac{\rho v}{1+\rho} |\la|}{\tanh 2\pi \frac{\rho v}{1+\rho} |\la|}-1\right)\right]\times\\
&\times\left(\frac{2\pi \frac{\rho v}{1+\rho} |\la|}{\sinh 2\pi \frac{\rho v}{1+\rho} |\la|}\right)^{1-s}\left(\frac{2\pi \frac{v}{1+\rho}|\la|}{\sinh 2\pi \frac{ v}{1+\rho}|\la|}\right)^{1+s} \left(\frac{|\la|}{2\sinh 2\pi v |\la|}\right)^{\frac m2}  e^{-\frac \pi2 |z'|^2 \frac{|\la|}{\tanh 2\pi v |\la|}}d\la dg'.
\end{align*}
Hence, recalling the definition of $A(s,\mu)$ in \eqref{monstersolved}, we deduce
\begin{align*}
&\int_0^\infty\int_0^\infty \frac{\rho^{s}}{v}\frac{\partial}{\partial\rho}P^\rho_{(-s,s),v}u (g) dvd\rho\\
&=\int_0^\infty\int_{\bG}\int_{\R^k}\frac{1}{v}u(g\circ g')e^{2\pi i \langle \sigma',\la\rangle} \left(\frac{|\la|}{2\sinh 2\pi v |\la|}\right)^{\frac m2}  e^{-\frac \pi2 |z'|^2 \frac{|\la|}{\tanh 2\pi v |\la|}} A(s,2\pi v|\la|)d\la dg' dv\\
&=0,
\end{align*}
where in the last equality we have used Lemma \ref{lemmaMonster}. Therefore, from \eqref{hereweare}, we infer
$$
\mathscr I_{(2s)} (\mathscr L_s u)(g)=u(g).
$$
The proof of the reversed relation $\mathscr L_s( \mathscr I_{(2s)} u)=u$ is completely analogous and it corresponds formally to the change $s\mapsto -s$ in the above proof (keeping in mind the commutation relation $P_{(s),t}P_{(-s),\tau}=P_{(-s),\tau}P_{(s),t}$ and the fact that $A(s,\mu)$ vanishes also for $s\in (-1,0)$ by Lemma \ref{lemmaMonster}). 
 
\end{proof}


 

\section{Unraveling the kernels}\label{S:one}

In this final section we prove Theorems \ref{T:poisson} and \ref{T:thick}. To motivate our first set of
results we observe that a direct consequence of Proposition \ref{P:Isuno} is that the kernel $\mathscr E^{(s)}(g)$ in  \eqref{EsH0}
 constitutes the fundamental solution of the nonlocal operator $\mathscr L^s$ with pole at the group identity.


\subsection{Non-conformal fundamental solution}

In our first result, Theorem \ref{T:fss} below, we provide an explicit integral expression for such kernel in the setting of groups of Heisenberg type. In discrepancy with Theorem \ref{T:poisson} and Corollary \ref{C:Htypes} such result shows in particular that, in the logarithmic coordinates $g = (z,\sigma)\in \bG$, one has $\mathscr E^{(s)}(z,\sigma) = \Phi(|z|^4,|\sigma|^2)$, but gauge symmetry breaks down when $0<s<1$.  
In what follows the notation $F(\alpha,\beta;\gamma;x)$ indicates Gauss hypergeometric function, see Section \ref{S:prelim}.

\begin{theorem}\label{T:fss}
Let $\bG$ be a group of Heisenberg type. For any $0<s< 1$ the fundamental solution \eqref{EsH0} of the operator $\mathscr L^s$ with pole at the group identity is given by the formula
\begin{align}\label{EsH5}
\mathscr E^{(s)}(z,\sigma) & = \frac{2^{k- 2s}\G(\frac m2 + k - s)}{\pi^{\frac{m+k}2}\Gamma(s)\G(\frac k2)} \frac{1}{ |z|^{2(\frac m2 + k - s)}} \int_0^1 (\tanh^{-1} \sqrt y)^{s-1} \left(1-y\right)^{\frac m4-1}  y^{\frac 12(k - s-1)}
\\
& \times  F\left(\frac{1}{2}(\frac m2 + k - s),\frac 12(\frac m2 + k+1-s);\frac k2;- \frac{16|\sigma|^2}{|z|^4} y\right) \ dy. 
\notag
\end{align}
In particular, such function depends only on $|z|^4$ and $|\sigma|^2$, but it is not a function of the gauge $N(z,\sigma) = (|z|^4 + 16 |\sigma|^2)^{1/4}$.
\end{theorem}

\begin{proof}
Our starting point is the formula \eqref{EsH0}  
which defines the fundamental solution of the operator $\mathscr L^{s}$. Inserting the expression \eqref{pipposemprepiupiccoloHoo} in \eqref{EsH0} we find
\begin{align}\label{EsH}
\mathscr E^{(s)}(z,\sigma) & = \frac{2^k}{(4\pi)^{\frac m2+k}\Gamma(s)} \int_0^\infty \frac{1}{t^{\frac m2+k -s}} \int_{\R^k} \left(\frac{|\la|}{\sinh |\la|}\right)^{\frac m2} e^{-\frac it \langle \sigma,\la\rangle} e^{- \frac{|z|^2}{4t} \frac{|\la|}{\tanh |\la|}}d\la \frac{dt}t
\\
& = \frac{2^k}{(4\pi)^{\frac m2+k}\Gamma(s)} \int_0^\infty t^{\frac m2+k -s} \int_{\R^k} \left(\frac{|\la|}{\sinh |\la|}\right)^{\frac m2} e^{-i t \langle \sigma,\la\rangle} e^{- t \frac{|z|^2}{4} \frac{|\la|}{\tanh |\la|}}d\la \frac{dt}t.
\notag
\end{align}
Since we want to use Proposition \ref{P:FTspheren} we assume hereafter that $k\ge 2$ and leave to the reader to provide the simple modifications for the case $k=1$. 
Using Cavalieri's principle, we find
\begin{align*}
& \int_{\R^k} \left(\frac{|\la|}{\sinh |\la|}\right)^{\frac m2} e^{-i t \langle \sigma,\la\rangle} e^{- t \frac{|z|^2}{4} \frac{|\la|}{\tanh |\la|}}d\la = \int_0^\infty \left(\frac{r}{\sinh r}\right)^{\frac m2} e^{- t \frac{|z|^2}{4} \frac{r}{\tanh r}} 
\\
& \times \left(\int_{\mathbb S^{k-1}} e^{-i t r\langle \sigma,\omega\rangle} d\omega\right) r^{k-1} dr.
\end{align*}
Applying Proposition \ref{P:FTspheren} with $\xi = \frac{t r \sigma}{2\pi}$, we obtain 
\begin{equation}\label{units}
\int_{\mathbb S^{k-1}} e^{-i t r\langle \sigma,\omega\rangle} d\omega = (2\pi)^{\frac k2} t^{-\frac k2 + 1} r^{-\frac k2 + 1} |\sigma|^{-\frac k2 + 1}J_{\frac{k}2-1}(t r |\sigma|).
\end{equation}
Inserting \eqref{units} in the above formula gives
\begin{align*}
& \int_{\R^k} \left(\frac{|\la|}{\sinh |\la|}\right)^{\frac m2} e^{-i t \langle \sigma,\la\rangle} e^{- t \frac{|z|^2}{4} \frac{|\la|}{\tanh |\la|}}d\la
\\
& = (2\pi)^{\frac k2} |\sigma|^{-\frac k2 + 1} t^{-\frac k2 + 1} \int_0^\infty r^{\frac k2} \left(\frac{r}{\sinh r}\right)^{\frac m2} e^{- t \frac{|z|^2}{4} \frac{r}{\tanh r}}  J_{\frac{k}2-1}(t r |\sigma|) dr.
\end{align*}
Returning with this information to \eqref{EsH} we find
\begin{align}\label{EsH2}
\mathscr E^{(s)}(z,\sigma) & = \frac{2^k (2\pi)^{\frac k2} |\sigma|^{-\frac k2 + 1}}{(4\pi)^{\frac m2+k}\Gamma(s)} \int_0^\infty t^{\frac m2+\frac k2 -s} \int_0^\infty r^{\frac k2} \left(\frac{r}{\sinh r}\right)^{\frac m2} e^{- t \frac{|z|^2}{4} \frac{r}{\tanh r}}  J_{\frac{k}2-1}(t r |\sigma|) dr dt
\\
& = \frac{2^k (2\pi)^{\frac k2}|\sigma|^{-\frac k2 + 1}}{(4\pi)^{\frac m2+k}\Gamma(s)} \int_0^\infty r^{\frac k2} \left(\frac{r}{\sinh r}\right)^{\frac m2} \int_0^\infty  t^{\frac m2+\frac k2 -s}  e^{- t \frac{|z|^2}{4} \frac{r}{\tanh r}}  J_{\frac{k}2-1}(t r |\sigma|)  dt\ dr.
\notag
\end{align}
We now apply formula \eqref{ohoh} with the choice
\[
\nu = \frac k2 -1,\ \ \ \mu = \frac m2 + \frac k2 - s + 1,\ \ \ \alpha = \frac{|z|^2}{4} \frac{r}{\tanh r},\ \ \ \beta = r |\sigma|.
\]
This gives
\[
\nu + \mu = \frac m2 + k - s,\ \ \ \ 1-\mu+\nu = -\frac m2 - 1 + s.
\]
Notice that, since $m\ge 2, k\ge 1$, we have $\nu+\mu\ge 2-s>0$. Furthermore,
\[
\alpha^2 + \beta^2 = \frac{r^2}{16 \tanh^2 r} (|z|^4 + 16 |\sigma|^2 \tanh^2 r),\ \ \ \frac{\beta^2}{\alpha^2 + \beta^2} = \frac{16 |\sigma|^2 \tanh^2 r}{|z|^4 + 16 |\sigma|^2 \tanh^2 r}.
\]
We thus obtain
\begin{align}\label{magicone}
& \int_0^\infty  t^{\frac m2+\frac k2 -s}  e^{- t \frac{|z|^2}{4} \frac{r}{\tanh r}}  J_{\frac{k}2-1}(t r |\sigma|)  dt  = \frac{2^{1-\frac k2} (r|\sigma|)^{\frac k2 -1} \G(\frac m2 + k - s)}{\G(\frac k2) (|z|^4 + 16 |\sigma|^2 \tanh^2 r)^{\frac{1}2(\frac m2 + k - s)}} 
\\
& \times \left(\frac{16 \tanh^2 r}{r^2}\right)^{\frac{1}2(\frac m2 + k - s)} F\left(\frac{1}{2}(\frac m2 + k - s),-\frac{1}{2}(\frac m2 + 1 - s);\frac k2;\frac{16 |\sigma|^2 \tanh^2 r}{|z|^4 + 16 |\sigma|^2 \tanh^2 r}\right)
\notag
\\
& = \frac{4^{\frac m2 + k - s}2^{1-\frac k2} |\sigma|^{\frac k2 -1}  \G(\frac m2 + k - s)}{\G(\frac k2) } \frac{r^{\frac k2 -1} (\tanh^2 r)^{\frac{1}2(\frac m2 + k - s)}}{r^{\frac m2 + k - s} (|z|^4 + 16 |\sigma|^2 \tanh^2 r)^{\frac{1}2(\frac m2 + k - s)}}
\notag
\\
& \times F\left(\frac{1}{2}(\frac m2 + k - s),-\frac{1}{2}(\frac m2 + 1 - s);\frac k2;\frac{16 |\sigma|^2 \tanh^2 r}{|z|^4 + 16 |\sigma|^2 \tanh^2 r}\right).
\notag
\end{align}
Substituting \eqref{magicone} in \eqref{EsH2} we find
\begin{align}\label{EsH33}
\mathscr E^{(s)}(z,\sigma) & = \frac{2^k (2\pi)^{\frac k2}}{(4\pi)^{\frac m2+k}\Gamma(s)} \frac{4^{\frac m2 + k - s}2^{1-\frac k2} \G(\frac m2 + k - s)}{2\G(\frac k2) } \int_0^\infty r^{s-1} \left(\frac{1}{\cosh^2 r}\right)^{\frac m4-1} 
\\
& \times \frac{(\tanh^2 r)^{\frac 12{(k - s-1)}}}{(|z|^4 + 16 |\sigma|^2 \tanh^2 r)^{\frac{1}2(\frac m2 + k - s)}}\ \frac{2\tanh r}{\cosh^2 r}
\notag\\
& \times  F\left(\frac{1}{2}(\frac m2 + k - s),-\frac{1}{2}(\frac m2 + 1 - s);\frac k2;\frac{16 |\sigma|^2 \tanh^2 r}{|z|^4 + 16 |\sigma|^2 \tanh^2 r}\right) dr.
\notag 
\end{align}
As a help to the reader, we mention that the presence of the factor $r^{s-1}$ in the integral in the right-hand side of \eqref{EsH33} is the reason for the break in gauge symmetry. 
To proceed we now use the formula \eqref{hyperG} with the choices
\[
\alpha = \frac{1}{2}(\frac m2 + k - s),\ \ \ \gamma = \frac k2,\ \ \ \beta = \frac 12(\frac m2 + k +1 -s),\ \ \ \frac{u}{u-1} = \frac{16 |\sigma|^2 \tanh^2 r}{|z|^4 + 16 |\sigma|^2 \tanh^2 r}.
\]
Notice that, with these choices, we have $\gamma - \beta = - \frac 12(\frac m2+1-s)$, and that 
\[
u = - \frac{16|\sigma|^2}{|z|^4} \tanh^2 r,\ \ \ \ \ 1-u = \frac{|z|^4 + 16 |\sigma|^2 \tanh^2 r}{|z|^4}.
\] 
We thus find from \eqref{hyperG}
\begin{align}\label{holyminchias}
& F\left(\frac{1}{2}(\frac m2 + k - s),-\frac 12(\frac m2 +1 -s);\frac k2;\frac{16 |\sigma|^2 \tanh^2 r}{|z|^4 + 16 |\sigma|^2 \tanh^2 r}\right) 
\\
& = \frac{(|z|^4 + 16 |\sigma|^2 \tanh^2 r)^{\frac{1}{2}(\frac m2 + k - s)}}{|z|^{2(\frac m2 + k - s)}} 
\notag \\
 & \times F\left(\frac{1}{2}(\frac m2 + k - s),\frac 12(\frac m2 + k+1-s);\frac k2;- \frac{16|\sigma|^2}{|z|^4} \tanh^2 r\right).
\notag
\end{align}
Substituting \eqref{holyminchias} in \eqref{EsH33} we find
\begin{align}\label{EsH4}
\mathscr E^{(s)}(z,\sigma) & = \frac{2^k \pi^{\frac k2}}{\pi^{\frac m2+k}\Gamma(s)} \frac{2^{- 2s} \G(\frac m2 + k - s)}{\G(\frac k2) |z|^{2(\frac m2 + k - s)}} \int_0^\infty r^{s-1} \left(1-\tanh^2 r\right)^{\frac m4-1} (\tanh^2 r)^{\frac 12{(k - s-1)}}
\\
& \times F\left(\frac{1}{2}(\frac m2 + k - s),\frac 12(\frac m2 + k+1-s);\frac k2;- \frac{16|\sigma|^2}{|z|^4} \tanh^2 r\right) \ \frac{2\tanh r}{\cosh^2 r} dr.
\notag 
\end{align}
In \eqref{EsH4} we now make the change of variable $y = \tanh^2 r$, which gives $dy = \frac{2\tanh r}{\cosh^2 r} dr$, obtaning
\begin{align*}
\mathscr E^{(s)}(z,\sigma) & = \frac{2^k \pi^{\frac k2}}{\pi^{\frac m2+k}\Gamma(s)} \frac{2^{- 2s} \G(\frac m2 + k - s)}{\G(\frac k2) |z|^{2(\frac m2 + k - s)}} \int_0^1 (\tanh^{-1} \sqrt y)^{s-1} \left(1-y\right)^{\frac m4-1} y^{\frac 12(k - s-1)}
\\
& \times F\left(\frac{1}{2}(\frac m2 + k - s),\frac 12(\frac m2 + k+1-s);\frac k2;- \frac{16|\sigma|^2}{|z|^4} y\right) \ dy. 
\notag
\end{align*}
This proves \eqref{EsH5} thus completing the proof.

\end{proof}

\begin{remark}\label{sinmenos}
We emphasise that formula \eqref{EsH5} continues to be valid if $s$ is changed into $-s$, provided that $\G(s)$ is replaced by $|\G(-s)|$. In such case, the resulting kernel 
\[
\mathscr E^{(-s)}(g,g') = - \frac{s}{\G(1-s)}\int_0^\infty \frac{1}{t^{1+s}} p(g,g',t) dt
\]
provides the following Riesz type representation for \eqref{bala}
\[
(-\mathscr L)^s u(g) = \operatorname{PV} \int_{\bG} \mathscr E^{(-s)}(g,g') \big[u(g) - u(g')\big] dg'.
\]
\end{remark}

The reader should notice  that, as expected from \eqref{EsH0}, the function $\mathscr E^{(s)}(z,\sigma)$ in \eqref{EsH5} is homogeneous of degree $\kappa = 2s- Q$ with respect to the anisotropic group dilations $(z,\sigma)\to (\la z,\la^2\sigma)$, where $Q = m+2k$ is the homogeneous dimension of $\bG$. One has in fact from \eqref{EsH5}
\[
\mathscr E^{(s)}(\la z,\la^2 \sigma) = \la^{2s -Q} \mathscr E^{(s)}(z,\sigma).
\]
 We mention that in the special setting of the Heisenberg group $\Hn$ Theorem \ref{T:fss} was proved with a different approach from ours in \cite[Proposition 4.1]{AM} by Askour and Mouayn. Still in $\Hn$, formulas of the fundamental solutions of the powers $\mathscr L^p$, where $p\in \mathbb N$, were found by Benson, Dooley and Ratcliff in \cite{BDR}.


\subsection{Conformal fundamental solutions}

We next present the counterpart of Theorem \ref{T:fss} for the operator $\mathscr L_s$. This is where geometry appears in the  form of the modified heat kernel \eqref{poisson}.

\begin{proof}[Proof of Theorem \ref{T:poisson}]
Proceeding as in the proof of Theorem \ref{T:fss}, from the expression of the conformal heat kernel \eqref{poisson} and from the definition \eqref{Esottos}, we have
\begin{align}\label{KsH}
\mathscr E_{(s)}(z,\sigma) &  = \frac{2^k}{(4\pi)^{\frac m2+k}\G(s)} \int_0^\infty t^{\frac m2+k -s-1}   \int_{\R^k} e^{-i t \langle \sigma,\la\rangle} \left(\frac{|\la|}{\sinh |\la|}\right)^{\frac m2+1-s}  e^{- t \frac{|z|^2}{4} \frac{|\la|}{\tanh |\la|}}d\la dt.
\end{align}
Cavalieri's principle and \eqref{units} give as before
\begin{align*}
& \int_{\R^k} e^{-i t \langle \sigma,\la\rangle} \left(\frac{|\la|}{\sinh |\la|}\right)^{\frac m2+1-s}  e^{- t \frac{|z|^2}{4} \frac{|\la|}{\tanh |\la|}}d\la
\\
& = (2\pi)^{\frac k2} |\sigma|^{-\frac k2 + 1} t^{-\frac k2 + 1} \int_0^\infty r^{\frac k2} \left(\frac{r}{\sinh r}\right)^{\frac m2+1-s} e^{- t \frac{|z|^2}{4} \frac{r}{\tanh r}}  J_{\frac{k}2-1}(t r |\sigma|) dr.
\end{align*}
Inserting this result in \eqref{KsH} gives
\begin{align}\label{KsH2}
& \mathscr E_{(s)}(z,\sigma)  
= \frac{2^k (2\pi)^{\frac k2}|\sigma|^{-\frac k2 + 1}}{(4\pi)^{\frac m2+k}\G(s)} 
 \int_0^\infty r^{\frac k2} \left(\frac{r}{\sinh r}\right)^{\frac m2+1-s} 
 \\
 & \times \int_0^\infty  t^{\frac m2+\frac k2 -s}  e^{- t \frac{|z|^2}{4} \frac{r}{\tanh r}}\  J_{\frac{k}2-1}(t r |\sigma|)  dt\ dr.
\notag\end{align}
At this point, if as in the proof of Theorem \ref{T:fss} we substitute in \eqref{KsH2} the formula \eqref{magicone}, we see that a miracle happens: the powers of $r$ cancel, and we obtain the expression
\begin{align}\label{m1}
\mathscr E_{(s)}(z,\sigma)    & = \frac{2^{k-2s+1} \G(\frac m2 + k - s)}{\pi^{\frac{m+k}2}\G(\frac k2) \G(s)} \int_0^\infty \left(\frac{1}{\sinh^2 r}\right)^{\frac 12(\frac m2+1-s)} \frac{(\tanh^2 r)^{\frac 12(\frac m2 + k-s)}}{(|z|^4 + 16 |\sigma|^2 \tanh^2 r)^{\frac 12(\frac m2 + k -s)}}
\\
& \times   F\left(\frac{1}{2}(\frac m2 + k - s),-\frac{1}{2}(\frac m2 + 1 - s);\frac k2;\frac{16 |\sigma|^2 \tanh^2 r}{|z|^4 + 16 |\sigma|^2 \tanh^2 r}\right) dr.
\notag
\end{align}
We now use \eqref{holyminchias} in \eqref{m1} to find
\begin{align}\label{m2}
\mathscr E_{(s)}(z,\sigma)    & = \frac{2^{k-2s+1} \G(\frac m2 + k - s)}{\pi^{\frac{m+k}2}\G(\frac k2) \G(s)|z|^{2(\frac m2 + k - s)}} \int_0^\infty \left(\frac{1}{\sinh^2 r}\right)^{\frac 12(\frac m2+1-s)} (\tanh^2 r)^{\frac 12(\frac m2 + k-s)}
\\
& \times  F\left(\frac{1}{2}(\frac m2 + k - s),\frac 12(\frac m2 + k+1-s);\frac k2;- \frac{16|\sigma|^2}{|z|^4} \tanh^2 r\right) dr.
\notag
\end{align}
Next, to prepare the situation for  the change of variable $y = \tanh^2 r$ in the integral in the right-hand side of \eqref{m2}, we rearrange the terms  in the following way:
\begin{align}\label{m3}
& \left(\frac{1}{\sinh^2 r}\right)^{\frac 12(\frac m2+1-s)} (\tanh^2 r)^{\frac 12(\frac m2 + k-s)} = \frac 12 \left(\frac{1}{\cosh^2 r}\right)^{\frac 12(\frac m2-1-s)} (\tanh^2 r)^{\frac k2 -1}\ \frac{2 \tanh r}{\cosh^2 r}
\\
& = \frac 12 \left(1-\tanh^2 r\right)^{\frac 12(\frac m2-1-s)} (\tanh^2 r)^{\frac k2 -1}\ \frac{2 \tanh r}{\cosh^2 r}.
\notag
\end{align}
Substituting \eqref{m3} in \eqref{m2} we obtain
\begin{align}\label{KsH4}
\mathscr E_{(s)}(z,\sigma)  & = \frac{2^{k-2s} \G(\frac m2 + k - s)}{\pi^{\frac{m+k}2}\G(\frac k2)\G(s) |z|^{2(\frac m2 + k - s)}} \int_0^\infty \left(1-\tanh^2 r\right)^{\frac 12(\frac m2-1-s)} (\tanh^2 r)^{\frac k2 -1}
\\
& \times F\left(\frac{1}{2}(\frac m2 + k - s),\frac 12(\frac m2 + k+1-s);\frac k2;- \frac{16|\sigma|^2}{|z|^4} \tanh^2 r\right) \ \frac{2\tanh r}{\cosh^2 r} dr.
\notag 
\end{align}
If in \eqref{KsH4} we now make the change of variable $y = \tanh^2 r$, which gives $dy = \frac{2\tanh r}{\cosh^2 r} dr$, we find
\begin{align}\label{lollypop}
\mathscr E_{(s)}(z,\sigma)   & = \frac{2^{k-2s} \G(\frac m2 + k - s)}{\pi^{\frac{m+k}2}\G(\frac k2)\G(s) |z|^{2(\frac m2 + k - s)}} \int_0^1 \left(1-y\right)^{\frac 12(\frac m2-1-s)} y^{\frac k2 -1}
\\
& \times F\left(\frac{1}{2}(\frac m2 + k - s),\frac 12(\frac m2 + k+1-s);\frac k2;- \frac{16|\sigma|^2}{|z|^4} y\right) \ dy. 
\notag
\end{align}
At this point we use Bateman's formula \eqref{hyperGint}, in which we choose
\[
c = \frac k2,\ \ \ \gamma = \frac 12(\frac m2 + k+1-s),\ \ \ \alpha = \frac{1}{2}(\frac m2 + k - s),\ \ \ \beta = \frac 12(\frac m2 + k+1-s),\ \ \ a = - \frac{16|\sigma|^2}{|z|^4}.
\]
This gives
\begin{align}\label{K3}
& \int_0^1 \left(1-y\right)^{\frac 12(\frac m2-1-s)} y^{\frac k2 -1} F\left(\frac{1}{2}(\frac m2 + k - s),\frac 12(\frac m2 + k+1-s);\frac k2;- \frac{16|\sigma|^2}{|z|^4} y\right)  dy
\\
& = \frac{\G(\frac k2)\G(\frac 12(\frac m2+1-s))}{\G(\frac 12(\frac m2 + k+1-s))} F(\frac{1}{2}(\frac m2 + k - s),\frac 12(\frac m2 + k+1-s);\frac 12(\frac m2 + k+1-s);- \frac{16|\sigma|^2}{|z|^4}).
\notag
\end{align}
This is a remarkable conclusion since we can now take advantage of formula \eqref{fs6} to infer that
\[
F(\frac{1}{2}(\frac m2 + k - s),\frac 12(\frac m2 + k+1-s);\frac 12(\frac m2 + k+1-s);- \frac{16|\sigma|^2}{|z|^4}) = \frac{|z|^{2(\frac m2 + k - s)}}{(|z|^4 + 16 |\sigma|^2)^{\frac{1}{2}(\frac m2 + k - s)}}.
\]
Substitution of this result in \eqref{K3} gives
\begin{align}\label{K4}
&  \int_0^1  \left(1-y\right)^{\frac 12(\frac m2-1-s)} y^{\frac k2 -1}   F\left(\frac{1}{2}(\frac m2 + k - s),\frac 12(\frac m2 + k+1-s);\frac k2;- \frac{16|\sigma|^2}{|z|^4} y\right) dy
\\
& = \frac{\G(\frac k2)\G(\frac 12(\frac m2+1-s))}{\G(\frac 12(\frac m2 + k+1-s))} |z|^{2(\frac m2 + k - s)}\ \left(|z|^4 + 16 |\sigma|^2\right)^{-\frac{1}{2}(\frac m2 + k - s)}.
\notag
\end{align}
Combining \eqref{lollypop} with \eqref{K4} we obtain
\[
\mathscr E_{(s)}(z,\sigma)   = \frac{2^{k-2s} \G(\frac m2 + k - s)\G(\frac 12(\frac m2+1-s))}{\pi^{\frac{m+k}2} \G(s) \G(\frac 12(\frac m2 + k+1-s))}  \left(|z|^4 + 16 |\sigma|^2\right)^{-\frac{1}{2}(\frac m2 + k-s)}.
\]
If we now use the Legendre duplication formula (see (1.2.3) on p.3 in \cite{Le})
\begin{equation}\label{prod}
2^{2x-1} \G(x) \G(x+\frac 12) = \sqrt \pi \G(2x),
\end{equation}
in which we take $x = \frac 12(\frac m2 + k -s)$, we find
\[
2^{\frac m2 + k -1-s} \G(\frac 12(\frac m2 + k -s)) \G(\frac 12(\frac m2 + k+1-s)) = \sqrt \pi \G(\frac m2 + k -s).
\]
Inserting this identity in the above formula we finally have
\[
\mathscr E_{(s)}(z,\sigma)   = \frac{2^{\frac m2 + 2k-3s-1} \G(\frac 12(\frac m2+1-s)) \G(\frac 12(\frac m2 + k -s))}{\pi^{\frac{m+k+1}2}\G(s)} \left(|z|^4 + 16 |\sigma|^2\right)^{-\frac{1}{2}(\frac m2 + k-s)}.
\]
Comparing the latter expression with \eqref{Esottos}, and keeping \eqref{C} in mind, we see  that the proof is completed.

\end{proof}

Before proceeding we note that the poles of the constant $C_{(s)}(m,k)$ in \eqref{C} occur for values of $s$ such that
\[
s = \frac m2 + 1 + 2j,\ \ \ j\in \mathbb N\cup\{0\},\ \ \text{or}\ \ \ s = \frac m2 + k + 2\ell,\ \ \ \ell\in \mathbb N\cup\{0\}
\]
(when $k =1$ and we are thus in the framework of $\Hn$, these two exceptional sets coincide). Therefore, the function $\mathscr E_{(s)}(z,\sigma)$ continues to provide a fundamental solution for the higher-order fractional powers of $\mathscr L_s$ in the range $0<s< \frac m2 + 1$. One should also see  \cite[Proposition A.1]{BFM}. Although these authors deal with the conformal Laplacian on the sphere in $\mathbb C^{n+1}$, using the Cayley map one can extract from their statement the Heisenberg group case of Theorem \ref{T:poisson}. 

\medskip

Finally, we present the 

\begin{proof}[Proof of Theorem \ref{T:thick}]
It follows closely the lines of that of Theorem \ref{T:poisson}, therefore we will skip all the unnecessary details. We start from definition \eqref{frake} in which, in the integral in the right-hand side, we substitute the expression of $q_{(s)}((z,\sigma),t,y)$ from \eqref{parBfs0}. After changing $t$ into $t^{-1}$, we find
\begin{align*}
& \mathfrak e_{(s)}(z,\sigma,y) = \frac{2^k}{(4\pi)^{\frac{m}2 +k +1-s}} \int_0^\infty t^{\frac m2+k -s-1}   \int_{\R^k} e^{-  i t \langle \sigma,\la\rangle}   \left(\frac{|\la|}{\sinh |\la|}\right)^{\frac m2+1-s} e^{-t \frac{|z|^2 +y^2}{4}\frac{|\la|}{\tanh |\la|}} d\la dt.
\end{align*}
Comparing the latter formula with \eqref{KsH}, we see that the two expressions differ exclusively by: 
\begin{itemize}
\item[(1)] the multiplicative factor $\frac{\G(s)}{(4\pi)^{1-s}}$; 
\item[(2)] the fact that the term $e^{- t \frac{|z|^2}{4} \frac{|\la|}{\tanh |\la|}}$ is replaced by $e^{-t \frac{|z|^2 +y^2}{4}\frac{|\la|}{\tanh |\la|}}$. 
\end{itemize}
Consequently, the final value of the integral expressing $\mathfrak e_{(s)}(z,\sigma,y)$ will be the same as the final value of \eqref{KsH}, except that there will be a multiplicative factor $\frac{\G(s)}{(4\pi)^{1-s}}$, and $|z|^4$ will be replaced by $(|z|^2 +y^2)^2$. This observation effectively finishes the proof of \eqref{qmasomeno}, and therefore of Theorem \ref{T:thick}.

\end{proof}

We mention in closing that the function $\mathfrak e_{(s)}((z,\sigma),y)$ in \eqref{qmasomeno} of Theorem \ref{T:thick} and its companion obtained by replacing $s$ with $-s$ are precisely those appearing in the formulas of Cowling and Haagerup in \cite[Section 3]{CH}, see also the works \cite{RTaim, RT}. In the case $s=1$, these functions were shown to play a central role in the study of the CR-Yamabe problem in  the Heisenberg group \cite{JL} (see also \cite{GV} for a partial result in groups of Heisenberg type), as well as in the Sobolev embedding theorem \cite{JL2} (see also the already cited work \cite{FL}).




\bibliographystyle{amsplain}

\end{document}